\documentclass[11pt]{article}

\usepackage{amssymb, amsmath}
\usepackage{amsthm}
\usepackage[affil-sl]{authblk}
\usepackage{multirow}
\usepackage{graphicx}

\newcommand\ie{{\rm\it{i.e.}}}
\newcommand\eg{{\rm\it{e.g.}}}
\newcommand\etal{{\rm\it{et al.}}}
\newcommand\Real{\mathbb{R}}
\newcommand{\dom}[1]{\textrm{dom}\,#1}
\newcommand{\Argmin}{\operatornamewithlimits{Argmin}}
\newcommand{\argmin}{\operatornamewithlimits{argmin}}
\newcommand{\norm}[1]{\|#1 \|}
\newcommand{\innprod}[2]{\left\langle#1,#2\right\rangle}
\newcommand\ds{\displaystyle}

\newcommand\w{w}
\newcommand\apsi{\psi}
\newcommand\z{z}
\newcommand\aphi{\varphi}
\newcommand\sz{z}
\newcommand\sphi{\varphi}

\newcommand\rel{R}
\newcommand\nrel{P}
\newcommand\srel{Q}

\newtheorem{thm}{Theorem}[section]{\bf}{\it}
\newtheorem{lem}[thm]{Lemma}{\bf}{\it}
{\bf}{\it}
\newtheorem{rem}[thm]{Remark}{\bf}{\rm}
{\bf}{\rm}
\newtheorem{defi}[thm]{Definition}{\bf}{\rm}
\newtheorem{algorithm}[thm]{Method}{\bf}{\rm}
\newtheorem{prope}{Property}{\bf}{\it}

\newtheorem{applem}{Lemma}[section]{\bf}{\it}

\newtheorem*{rem*}{Remark}{\bf}{\rm}

\newcommand{\SetMarginDefault}{
	\topmargin -0.5truein			
	\headheight 0.1truein 			
	\headsep 0.1truein 				
	\topskip 0.35truein 			
	\textheight 9.5 truein			
	\footskip 0.5 truein			
	
	\evensidemargin -0.1truein			
	\oddsidemargin -0.1truein			
	\textwidth 6.47truein 
	
}
\SetMarginDefault

\begin{document}


\title{New results on subgradient methods for strongly convex optimization problems with a unified analysis}
\author{Masaru Ito ({\it ito1@is.titech.ac.jp})}
\affil{Department of Mathematical and Computing Sciences,  Tokyo Institute of Technology\\2-12-1-W8-41 Oh-okayama, Meguro, Tokyo 152-8552 Japan}
\date{Research Report B-479\\Department of Mathematical and Computing Sciences\\Tokyo Institute of Technology\\April, 2015, revised December 2015}

\maketitle


\begin{abstract}
We develop subgradient- and gradient-based methods for minimizing strongly convex functions under a notion which generalizes the standard Euclidean strong convexity.
We propose a unifying framework for subgradient methods which yields two kinds of methods, namely, the Proximal Gradient Method (PGM) and the Conditional Gradient Method (CGM), unifying several existing methods.
The unifying framework provides tools to analyze the convergence of PGMs and CGMs for non-smooth, (weakly) smooth, and further for structured problems such as the inexact oracle models.
The proposed subgradient methods yield optimal PGMs for several classes of problems and yield optimal and nearly optimal CGMs for smooth and weakly smooth problems, respectively.
\\

\noindent
\textbf{Keywords:} non-smooth/smooth convex optimization, structured convex optimization, subgradient/gradient \hspace{-1.5truemm}-based proximal method, conditional gradient method, complexity theory, strongly convex functions, weakly smooth functions.
\\

\end{abstract}


\section{Introduction}

Subgradient- and gradient-based methods for convex optimization have been actively investigated in the last decades, providing efficient solutions for large-scale optimization problems which arise from image/signal processing, data mining, statistics, {\it etc}.
The efficiency of (sub)gradient-based methods are often analyzed from the viewpoint of oracle complexity \cite{NY79,Nes04} to ensure a given absolute accuracy $\varepsilon>0$ for the optimal value, and so far various ``optimal" methods are known for several classes of problems. Achieving the optimal complexity for subgradient methods usually requires a priori problem specific information; sometimes, however, we can attain optimal or nearly optimal complexity with less such requirements (but we may need some restrictions for their implementations).

The following two classes of convex problems have been particularly well studied: 
\begin{itemize}
\item {\it Non-smooth problems}. The problems of minimizing Lipschitz continuous convex functions with bounded subgradients;
\item {\it Smooth problems}. The problems of minimizing continuously differentiable convex functions with Lipschitz continuous gradients.
\end{itemize}
These two classes of convex problems can also be reformulated as {\it structured convex problems}, which have been receiving much attention in terms of both theoretical and application aspects. In particular, studies of (sub)gradient-based methods for the class of ``smoothable" functions \cite{ASS15,BT12,CJN13,Lan14a,Nes05s,Nes05e}, the class of composite problems \cite{ASS15,BT09,CLP12,FM81,GL12,GL13,Lan12,Nes13,Tseng08,Tseng10}, and the class of weakly smooth problems \cite{DGNs,DGN,Nes15u,Nes15} are notably important.

In this paper, we particularly focus on the following two kinds of (sub)gradient methods: the {\it Proximal (sub)Gradient Method (PGM)} and the {\it Conditional Gradient Method (CGM)}. Both methods may require easy-to-solve subproblems at each iteration.

The PGM is executed using a {\it prox-function} to define a reasonable proximal operator.
Based on the conceptual complexity of Nemirovski and Yudin \cite{NY79}, many important PGMs for the above classes of convex problems can be proposed and their optimal convergence can be achieved. As it will be pointed out in this paper, many of PGMs are modifications, accelerations, and/or combinations of two remarkably important PGMs, namely, the {\it Mirror-Descent Method (MDM)} \cite{BT03,NY79} and the {\it Dual-Averaging Method (DAM)} \cite{Nes09}, which are optimal for non-smooth problems.

The CGMs, on the other hand, are endowed by subproblems which are linear, {\it i.e.}, problems of minimizing a linear functional over a bounded convex feasible set.
Originating from Frank and Wolfe \cite{FW56}, convergence properties of CGMs are well analyzed (see \cite{DR70,DH78,FG14,Lan14a,Nes15,PD78} and references therein).
Because of their advantages such as easiness of subproblems and sparsity of approximate solutions, CGMs are actively studied with applications to machine learning and statistics \cite{CJN13,HJN15,Jag11,Jag13}; it is important to note that the CGMs have worse convergence rates than the PGMs, but the computational cost of each iteration of the former can be lower, compensating the overall cost. Therefore, it is extremely important to choose between the PGM or the CGM depending on the structure of the problem to solve.

In a recent work \cite{IF14}, a unifying framework of PGMs were proposed through a unifying treatment of the MDM and the DAM for non-smooth problems, and also for their corresponding accelerations \cite{Tseng08,Tseng10} for smooth (and structured) problems. 
This unifying framework enables one to generate a {\it family} of (optimal) subgradient methods which includes several existing optimal methods. Also it permits to analyze both the classical PGMs (\ie, the MDM and the DAM) for non-smooth problems and their accelerations for smooth problems under the same framework, whereas existing analysis for them were performed individually.
It is important to observe that  if we do not restrict the discussion to the MDM and the DAM, other universal optimal complexity methods were previously proposed for both non-smooth and smooth problems as well \cite{DGNs,DGN,GL12,GL13,Lan12,Nes15u}.

The work \cite{IF14}, however, focused only on PGMs and was developed without assuming the strong convexity of objective functions.
Using the knowledge of a strong convexity can help us to obtain much faster rate of convergence.
For instance, the MDM \cite{Bach15,NB,NL14} for non-smooth problems and Nesterov's PGMs \cite{Nes04,Nes13} for smooth (or composite) problems realize the optimal complexity in the strongly convex cases. Moreover, exploiting multistage procedures is a powerful approach to obtain optimal PGMs \cite{CLP12,GL13,JN14,NN85,Nes83,Nes13}.
However, the multistage procedures require a priori knowledge of an upper bound of the distance between the initial point and the optimal solution set.
Note that the optimal complexity of the DAM for non-smooth problems and of the Tseng's PGM for smooth problems are not known without the multistage procedure (see Sections \ref{ssec-existing-nonsm}, \ref{sssec-existing-struc}).

This paper proposes a new unifying framework of PGMs and CGMs for convex problems with strongly convex objective functions and its convergence analysis for both non-smooth and smooth problems.
The smooth problems become particular cases of structured problems by employing the generalized notion of the {\it inexact oracle model} \cite{DGNs,DGN}. It also enables us to handle simultaneously the weakly smooth problems.
The proposed methods require a priori knowledge of the convexity parameter of the objective function, while an upper bound for the distance between the initial point and the optimal solution set is not necessary to ensure the optimal convergence rate with respect to the iteration number.

We emphasize three particular contributions of this paper.

At first, the unifying framework yields generalizations of the MDM and the DAM originally proposed for non-smooth problems, and of Nesterov's and Tseng's optimal PGMs originally proposed for smooth (or composite) problems. As a consequence, the optimal convergence of the DAM and Tseng's PGMs for the strongly convex cases are new since the existing results were analyzed only for the non-strongly convex cases (Sections~\ref{ssec-eff-nonsm}, \ref{ssec-eff-modif-struc}). Our unifying framework also includes the {\it classical gradient methods} \cite{DGNs,Nes13} which were previously analyzed in the strongly convex case. However, our analysis provides a slightly improved convergence estimates for them (Section~\ref{ssec-eff-clas-struc}).

Secondly, a new family of CGMs can be obtained from the unifying framework, which includes the Lan's CGMs \cite{Lan14a}, and yields an optimal convergence result for smooth problems in the non-strongly convex case (Section~\ref{ssec-eff-modif-struc}); we further prove nearly optimal convergence rates of the proposed CGMs for the classes of weakly smooth problems (Section~\ref{sssec-CGM-weak}).  The advantage of our unifying framework is  a universal analysis of the PGMs and the CGMs.

Finally, we prove that our PGMs (including generalizations of Nesterov's and Tseng's PGMs) attains the optimal convergence rate for weakly smooth and strongly convex problems (and for further extended problems of the deterministic case of \cite{GL12}, Section~\ref{sssec-convergence-st}). We remark that the original Nesterov's and Tseng's PGMs were analyzed for smooth (or composite) problems only.
In contrast to the existing optimal method \cite{NN85}, our PGMs ensure the optimality with less a prior information for the objective function.

The current work can be seen as an extension of the recent work \cite{IF14}. The above mentioned three new contributions are particular consequences of the extension. In particular, the previous one \cite{IF14} can not consider the CGMs and the strongly convex cases. Moreover, we extended the structured problems of \cite{IF14} so that we can now handle weakly smooth problems efficiently.

Another extension from \cite{IF14} is that our framework (Property \ref{framework-double}) handles two kinds of auxiliary subproblems at each iterations which allows us to yield new variations of subgradient method including the Nesterov's method in \cite{Nes05s}.

This paper is organized as follows. We firstly discuss some general considerations about strongly convex problems in Section \ref{sec-background}. In particular, in Section \ref{ssec-problem-struc}, we introduce a kind of ``strong convexity" with respect to the prox-function and define the classes of {\it non-smooth} and of {\it structured} problems considered in this paper. We list some existing methods in the remaing part. We propose the {\it unified framework of subgradient-based methods} and general guidelines for constructing subproblems in Section \ref{sec-framework}. We analyze the proposed general (sub)gradient methods and establish general convergence results in Section \ref{sec-general-convergence}. Finally, in Section \ref{sec-convergence}, we discuss the rate of convergences for the non-smooth and the structured problems providing the (nearly) optimal complexity for them.

\section{Problem settings and existing methods}
\label{sec-background}


\subsection{Convex optimization problem and assumptions}
\label{ssec-problem-struc}

Let us consider the following convex optimization problem:
\begin{equation}\label{primal-problem}
\min_{x \in Q}f(x)
\end{equation}
where $Q$ is a closed convex subset of a finite dimensional real normed space $E$ equipped with a norm $\norm{\cdot}$, and $f:E \to \Real\cup\{+\infty\}$ is a lower-semicontinuous (lsc) convex function with $Q \subset \dom{f}$. We denote by $E^*$ the dual space of $E$ equipped with the dual norm $\norm{s}_* = \max_{\norm{x}\leq 1}\innprod{s}{x}$ for $s\in E^*$ where $\innprod{s}{x}$ is the value of $s \in E^*$ at $x \in E$.
We always assume that the problem (\ref{primal-problem}) has an optimal solution $x^* \in Q$.

Throughout this paper, we mainly focus on two particular classes of convex optimization problems (\ref{primal-problem}), the {\it class of non-smooth problems} and the {\it class of structured problems}, which will be defined shortly.

We introduce a {\it prox-function} $d(x)$ on the feasible set $Q$, that is, $d:E\to \Real\cup\{+\infty\}$ is a nonnegative, continuously differentiable, and strongly convex function on $Q$ (therefore, $Q \subset \dom{d}$) with a constant $\sigma_d > 0$ such that $d(x_0)=\min_{x \in Q}d(x)=0$ for the unique minimizer $x_0 \in Q$.
We use the notation $l_d(y;x):=d(y)+\innprod{\nabla{d}(y)}{x-y}$ for the linearization of $d(x)$ at $y \in Q$.
We also define the {\it Bregman distance} \cite{Breg} between $x$ and $y$ for $x,y \in Q$ by
\[ \xi(y,x) :=  d(x)-d(y)-\innprod{\nabla{d}(y)}{x-y} = d(x)-l_d(y;x).\]
Note that the strong convexity of $d(x)$ on $Q$ is equivalent to the property $\xi(y,x)\geq \frac{\sigma_d}{2}\norm{x-y}^2,~\forall x ,y \in Q$.
The prox-function as well as the Bregman distance will be used for the construction of auxiliary functions in the subproblems solved at each iterations in the methods described in this paper. We also assume that the prox-function $d(x)$ is fixed throughout the paper. A simple example for $d(x)$ is the {\it Euclidean setting}, namely, $E$ is a Euclidean space with $\norm{x}_2=\innprod{x}{x}^{1/2}$, and $d(x) = \frac{1}{2}\norm{x-x_0}_2^2$ for some $x_0 \in Q$.

For a lsc convex function $\psi:E\to\Real\cup\{+\infty\}$ with $Q \subset \dom{\psi}$,  we introduce the set 
\[\sigma(\psi):=\{\tau \geq 0 ~:~ \psi(x)-\tau d(x) \textrm{ is a lsc convex function on } Q\}.\]
The set $\sigma(\psi)$ corresponds to the set of ``convexity parameters" of $\psi(x)$ on $Q$ with respect to the prox-function $d(x)$.
In the Euclidean setting $d(x)=\frac{1}{2}\norm{x-x_0}_2^2$, the set $\sigma(\psi)$ is the set of convexity parameters of $\psi(x)$ in the usual sense.
Furthermore, in general, it can be shown that $\tau \in \sigma(\psi)$ if and only if the following inequality holds:
\begin{equation}\label{st-conv-Breg}
\psi(x) \geq \psi(y)+\psi'(y;x-y)+\tau\xi(y,x),\quad \forall x,y \in Q~ (\subset \textrm{dom}\psi),
\end{equation}
where $\psi'(x;d)=\lim_{\alpha\downarrow 0}\frac{\psi(x+\alpha d)-\psi(x)}{\alpha}~(x\in \dom{\psi},~d \in E)$
\footnote{Notice that the function $\varphi(x):=\psi(x)-\tau d(x)$ satisfies $\varphi'(y;x-y)=\psi'(y;x-y)-\tau\innprod{\nabla{d}(y)}{x-y}$, $\forall x,y \in Q$. Hence, the convexity of $\varphi(x)$ on $Q$ implies $\varphi(x)\geq \varphi(y)+\varphi'(y;x-y),\forall x,y \in Q$, which is equivalent to (\ref{st-conv-Breg}). Conversely, since $\psi'(y;x-y)\geq -\psi'(y;y-x)$ holds and so is true for $\varphi(\cdot)$ for $x,y \in Q$, (\ref{st-conv-Breg}) implies the two inequalities $\varphi(y)\geq \varphi(z)+\varphi'(z;y-z)$ and $\varphi(x)\geq \varphi(z)-\varphi'(z;z-x)$ for $x,y,z \in Q$. Since $\varphi'(y;\cdot)$ is positively homogeneous, the convexity of $\varphi(\cdot)$ on $Q$ follows by taking a convex combination of the two with $z=\alpha x + (1-\alpha)y,\alpha \in [0,1],x,y \in Q$.}.
This form is similar to the characterization of the usual strong convexity of $\psi(x)$ on $Q$ with constant $\tau\geq 0$: $\psi(x)\geq \psi(y)+ \psi'(y;x-y)+\frac{\tau}{2}\norm{x-y}^2,~\forall x,y \in Q$.
Therefore, $\tau \in \sigma(\psi)$ implies the usual strong convexity of $\psi(x)$ on $Q$ with constant $\tau\sigma_d$, since $\xi(y,x) \geq \frac{\sigma_d}{2}\norm{x-y}^2,~\forall x,y \in Q$. On the other hand, if the Bregman distance $\xi(y,x)$ {\it grows quadratically} on $Q$ with a constant $A > 0$ (see \cite{GL12}), \ie, $\xi(y,x)\leq \frac{A}{2}\norm{x-y}^2,~\forall x,y \in Q$, then the usual strong convexity of $\psi(x)$ on $Q$ with a constant $\tau\geq 0$ implies $\tau/A \in \sigma(\psi)$.

We assume a ``strong convexity" of the objective function $f(x)$ by supposing that $\sigma(f)\setminus\{0\}\not=\emptyset$.
However, in order to deal with several structured optimization problems as we will see in Section \ref{ssec-existing-struc}, we need to assume stronger conditions on the objective function as follows.
Let us assume that, for each $y \in Q$, there exists a lsc convex function $m_f(y;\cdot):E \to \Real\cup\{+\infty\}$ such that $m_f(y;x) \leq f(x)$ for all $x \in Q$; we call the function $m_f(y;x)$ a {\it lower approximation model of $f(x)$}. We further assume that there exists a convexity parameter $\sigma_f \geq 0$ such that
\begin{equation}\label{s-conv-assump}
\sigma_f \in \sigma(f)\cap\bigcap_{y \in Q}\sigma(m_f(y;\cdot)).
\end{equation}
Note that, since $f'(x^*;x-x^*)\geq 0$ holds for all $x \in Q$ by the optimality of $x^*$, the condition $\sigma_f\in\sigma(f)$ implies that $f(x)-f(x^*)\geq \sigma_f\xi(x^*,x)$ for all $x \in Q$.

The function $m_f(y;x)$ can be seen as a strongly convex lower approximation of $f(x)$ at $y \in Q$, and its construction depends on the problem structure.
Notice also that the condition (\ref{s-conv-assump}) is not as restrictive as it is apparent to be specially if the problem (\ref{primal-problem}) is provided by some structure. 

The convex optimization problem (\ref{primal-problem}) which we consider in this paper will be particularized into the following two classes for convenience.

\begin{defi} 
The {\it class of non-smooth problems} consists of convex optimization problems (\ref{primal-problem}) where we assume for each problem that we know a subgradient mapping $g(x) \in \partial{f}(x),~x \in Q$ and a convexity parameter $\sigma_f \in \sigma(f)$.
Then, we can naturally define its lower approximation model $m_f(\cdot;\cdot)$ by
\begin{equation}\label{nonsm-m_f}
m_f(y;x):=f(y)+\innprod{g(y)}{x-y}+\sigma_f\xi(y,x).
\end{equation}
Therefore, it satisfies (\ref{s-conv-assump}). Moreover, we assume that for every $s \in E^*$ and $\beta > 0$, the following optimization problem is solvable:
\begin{equation}\label{subprob-simple}
\min_{x \in Q}\{\innprod{s}{x}+\beta d(x)\}.
\end{equation}
This class of problems is denoted by $\mathcal{NSP}(g,\sigma_f)$.
\end{defi}

Notice that non-smooth problems satisfy the requirement (\ref{s-conv-assump}) because $m_f(y;x)-\sigma_f d(x)$ becomes an affine function. For convenience, we denote $g_k:=g(x_k) \in \partial{f}(x_k)$ for test points $x_k$.

\begin{defi} 
The {\it class of structured problems} consists of convex optimization problems (\ref{primal-problem}) where we assume for each problem that there exists $(m_f(\cdot;\cdot),\sigma_f,\bar\sigma_f,L(\cdot),\delta(\cdot,\cdot))$, {\it i.e.}, functions and constants, satisfying the inequality
\begin{equation}\label{struc-ineq}
f(x) \leq [m_f(y;x)-\bar\sigma_f\xi(y,x)]+\frac{{L}(y)}{2}\norm{y-x}^2 + \delta(y,x),\quad \forall x,y \in Q,
\end{equation}
where $m_f(\cdot;\cdot)$ is a lower approximation model of $f(\cdot)$ which admits (\ref{s-conv-assump}) for $\sigma_f \geq 0$, $\delta(y,\cdot)$ is a nonnegative convex function on $Q$ for $y \in Q$, $L(\cdot)\geq 0$, and $\bar\sigma_f \in [0,\sigma_f]$.
We further assume that for every $\beta \geq 0$,~$y \in E$ and $s\in E^*$, the optimization problems of the following form is efficiently solvable:
\begin{equation}\label{subprob-struc}
\min_{x \in Q}\{m_f(y;x) + \innprod{s}{x} + \beta d(x) \}.
\end{equation}
This class of problems is denote by $\mathcal{SP}(m_f,\sigma_f,\bar\sigma_f,L,\delta)$.
\end{defi}

Examples of such structured problems will be presented in Section \ref{sssec-examples-struc}.

The optimization problem (\ref{subprob-struc}) in the class of structured problems may differ from (\ref{subprob-simple}) in the class of non-smooth ones depending on how we choose the functions $m_f(\cdot;\cdot)$ (\eg, see the example (ii) in Section \ref{sssec-examples-struc}).

Note that when $\beta = 0$ and $\sigma_f=0$, problem (\ref{subprob-struc}) may be a minimization of a convex function which is non-strongly convex, in particular, an affine function on $Q$. In this case, we additionally assume the boundedness of $Q$ to ensure the existence of its solution. This is the case for the conditional gradient methods.

After developing a general analysis in Section \ref{sec-general-convergence}, the function $\delta(y,x)$ will be finally particularized for the constant case $\delta(y,x)\equiv \delta$ in Sections \ref{ssec-eff-clas-struc}, \ref{ssec-eff-modif-struc}, and for the case $\delta(y,x):=\frac{M}{\rho}\norm{x-y}^\rho,~M\geq 0,~\rho\in[1,2)$ in Section \ref{ssec-eff-modif-weak} (see Section \ref{ssec-existing-struc} for several examples and related works). Note that, when $\delta(y,x)\equiv\delta$ and $\sigma_f=0$, the structured problem is equivalent to the one introduced in \cite[Section 5]{IF14}.


\subsection{Existing methods for non-smooth problems}
\label{ssec-existing-nonsm}

Consider the non-smooth problems in the class $\mathcal{NSP}(g,\sigma_f)$. We assume for the moment that the subgradient mapping $g(x) \in \partial{f}(x)$ of $f(x)$ is bounded, \ie, there exists $M>0$ such that
\begin{equation}\label{subgrad-bounded}
\norm{g(x)}_* \leq M,\quad\forall x \in Q.
\end{equation}
Let us first consider the case $\sigma_f=0$. The original MDM and DAM, which solve this class of problems, are known to be optimal PGMs. Considering the notation in \cite[Method 9(a)]{IF14}, they are particular cases of the following procedure:
\begin{equation}\label{existing-SGM}
x_0:=\sz_{-1}:=\argmin_{x \in Q}d(x),\quad x_{k+1}:=\sz_k,\quad k\geq 0,
\end{equation}
where $\sz_k$ is the solution of the following fixed subproblem either from the {\it extended Mirror-Descent (MD) model}
\begin{equation}\label{eMD-subprob-nonst}
\min_{x \in Q}\{\lambda_k m_f(x_k;x)+\beta_kd(x)-\beta_{k-1}l_d(\sz_{k-1};x)\},
\end{equation}
or from the {\it Dual-Averaging (DA) model}
\begin{equation}\label{DA-subprob-nonst}
\min_{x \in Q}\left\{\sum_{i=0}^k\lambda_im_f(x_i;x)+\beta_kd(x)\right\},
\end{equation}
where $\{\lambda_k\}_{k\geq 0}$ and $\{\beta_k\}_{k \geq -1}$ are positive parameters called {\it weight} (or {\it step-size}) and {\it scaling parameters}, respectively; recall that $m_f(y;x)=f(y)+\innprod{g(y)}{x-y}$ by the definition (\ref{nonsm-m_f}) if $\sigma_f=0$.

The MDM, originally proposed by Nemirovski and Yudin \cite{NY79} and related to proximal subgradient methods by Beck and Teboulle \cite{BT03}, corresponds to the method (\ref{existing-SGM}) with the update (\ref{eMD-subprob-nonst}) letting $\beta_k \equiv 1$.
On the other hand, the method (\ref{existing-SGM}) with the update (\ref{DA-subprob-nonst}) yields the original DAM proposed by Nesterov \cite{Nes09}.
Tuning the scaling parameter $\{\beta_k\}$ enables us to obtain an efficient convergence rate (see \cite{IF14,Nes09}); for instance, taking $\lambda_k=1$ and $\beta_k=O(\sqrt{k})$ yields that $f(\hat{x}_k)-f(x^*) \leq O(1/\sqrt{k})$ where $\hat{x}_k:=\sum_{i=0}^k\lambda_ix_i/\sum_{i=0}^k\lambda_i$.
In this case, one needs the values $d(x^*)$ and $M$ to define $\lambda_k$ and/or $\beta_k$ to achieve the optimal iteration complexity $O(M^2d(x^*)/(\sigma_d\varepsilon^2))$ for an absolute accuracy $\varepsilon>0$.

When $\sigma_f>0$ is known, the extended MDM also admits the optimal complexity $O(M^2/(\sigma_d\sigma_f\varepsilon))$ for the strongly convex case by choosing
$
\lambda_k := \frac{2}{\sigma_f(k+2)},~ \beta_k:=1
$
(\cite[Theorem 1]{NL14}; see also \cite{Bach15,NB} for related results).
Moreover, it is proved that a multistage procedure for the DAM achieves the optimal complexity for problems of minimizing {\it uniformly convex functions}, a generalization of strongly convex ones, with further consideration in a stochastic setting \cite{JN14}.

As we mention next, an extended class of problems including non-smooth and smooth ones are considered in \cite{GL12,GL13,NN85,Nes15u} which propose optimal PGMs for these problems and therefore for the non-smooth problems as well.

\subsection{Examples and existing methods for structured problems}
\label{ssec-existing-struc}


\subsubsection{Examples of structured problems}
\label{sssec-examples-struc}

The class $\mathcal{SP}(m_f,\sigma_f, \bar\sigma_f, L, \delta)$ of structured problems introduced in Section \ref{ssec-problem-struc} includes several special convex problems that are also possibly non-smooth. We list some existing examples and results which can be discussed in this setting considering the requirements (\ref{s-conv-assump}) and (\ref{struc-ineq}).

\begin{enumerate}
\renewcommand{\labelenumi}{(\roman{enumi})}

\item
{\it Smooth problems.}
Suppose that $f(x)$ belongs to $C_L^{1,1}(Q)$; that is, $f(x)$ is continuously differentiable on $Q$ and $\nabla{f}(x)$ satisfies the Lipschitz condition on $Q$ with constant $L > 0$: $\norm{\nabla{f}(x)-\nabla{f}(y)}_* \leq L\norm{x-y},~\forall x,y \in Q$. When we know a constant $\sigma_f \in \sigma(f)$, we can define
\[ m_f(y;x) := f(y) + \innprod{\nabla{f}(y)}{x-y} + \sigma_f\xi(y,x) \]
to obtain (\ref{s-conv-assump}) and (\ref{struc-ineq}) with $L(\cdot):=L$, $\bar\sigma_f:=\sigma_f$, and $\delta(\cdot,\cdot):=0$. The corresponding subproblem (\ref{subprob-struc}) reduces to the form (\ref{subprob-simple}). 

The smooth problem with the Euclidean setting $d(x)=\frac{1}{2}\norm{x-x_0}_2^2$ is the most basic one among the examples here; in this case, the lower complexity bounds $O(\sqrt{Ld(x^*)/\varepsilon})$ for the case $\sigma_f=0$ and $O(\sqrt{L/\sigma_f}\log(1/\varepsilon))$ for the case $\sigma_f>0$ are known for an absolute accuracy $\varepsilon>0$. The first optimal PGM for the Euclidean case was proposed by Nesterov \cite{Nes83} and its variants were developed in \cite{Nes04}, and in \cite{AT06,Nes05s} for non-strongly convex cases.

CGMs are also considered for the smooth problems, which achieve the complexity $O(LR/\varepsilon)$ where $R:={\rm Diam}(Q)=\sup_{x,y\in Q}\norm{x-y}$ \cite{DR70,DH78,FW56,FG14,Lan14a,PD78}; excepting Lan's modified CGMs \cite{Lan14a}, all of these CGMs are based on the classical CGM \cite{FW56}, as we show in the algorithm (\ref{classical-CGM}).

\item
{\it Composite problems.}
Consider an objective function $f(x)$ of the form $f(x) = f_0(x)+\varPsi(x)$ where $f_0 \in C_L^{1,1}(Q)$ and $\varPsi(x)$ is a lsc convex function on $Q$ with a simple structure. If we know constants $\sigma_{f_0}\in \sigma(f_0)$ and $\sigma_\varPsi \in \sigma(\varPsi)$, then, we can take
\[m_f(y;x):=f_0(y)+\innprod{\nabla{f_0}(y)}{x-y} + \sigma_{f_0}\xi(y,x)+\varPsi(x)\]
from which (\ref{s-conv-assump}) and (\ref{struc-ineq}) hold with $\sigma_f:=\sigma_{f_0} + \sigma_\varPsi,~~L(\cdot):=L,~~\bar{\sigma}_f:=\sigma_{f_0}$, and $\delta(\cdot,\cdot):=0$.
There are many PGMs for this problem \cite{FM81,BT09,Nes13,Tseng08,Tseng10} and they provide the same iteration complexity as the lowest complexity for the smooth problems in the non-strongly convex case (excepting the work by Fukushima and Mine \cite{FM81} because they studied this model without assuming the convexity for $f_0(x)$). Nesterov \cite{Nes13} further proposed an optimal method for strongly convex composite problems in the Euclidean setting.
The smoothing technique proposed by Nesterov \cite{Nes05s} and its extension \cite{BT12} for a special form of $\varPsi(x)$ are also important because of their significant advantage in efficiency, which have further consideration in the strongly convex case \cite{Nes05e}.

A generalization of CGM to the composite problems was investigated in \cite{ASS15,Bach15} which also deal with a duality relationship to the MDM.

\item
{\it Inexact oracle model}.
Suppose that $f(x)$ is equipped with a {\it first-order $(\delta,L,\mu)$-oracle} \cite{DGNs}, \ie, for each $y \in Q$, we can compute $(f_{\delta,L,\mu}(y), g_{\delta,L,\mu}(y)) \in \Real\times E^*$ such that
\[ \frac{\mu}{2}\norm{x-y}^2 \leq f(x)-(f_{\delta,L,\mu}(y)+\innprod{g_{\delta,L,\mu}(y)}{x-y}) \leq \frac{L}{2}\norm{x-y}^2+\delta,\quad \forall x \in Q,\]
where $\delta \geq 0$ and $L\geq \mu \geq 0$.
If $\mu=0$ or the prox-function grows quadratically on $Q$ with constant $A > 0$, then defining
\[ m_f(y;x):=f_{\delta,L,\mu}(y)+\innprod{g_{\delta,L,\mu}(y)}{x-y}+\frac{\mu}{A}\xi(y,x), \]
admits (\ref{s-conv-assump}) and (\ref{struc-ineq}) with $L(\cdot):=L$, $\sigma_f:=\bar\sigma_f:=\mu/A$, and $\delta(\cdot,\cdot):=\delta$.
The inexact oracle model with $\mu=0$ was firstly studied by Devolder \etal\, \cite{DGN} and they proposed the classical and the fast (proximal) gradient methods which were extended to the strongly convex case in \cite{DGNs}. A CGM for this model in the case $\mu=0$ was analyzed by \cite{FG14}.

\item
{\it Weakly smooth problems.}
Suppose that the objective function $f(x)$ belongs to $C_M^{1,\nu}(Q)$ for some $\nu \in [0,1)$, \ie, $f(x)$ is continuously differentiable on $Q$ and $\nabla{f}(x)$ satisfies the H\"older condition $\norm{\nabla{f}(x)-\nabla{f}(y)}_*\leq M\norm{x-y}^{\nu},~\forall x,y \in Q$; but in the case $\nu=0$, we do not assume the smoothness for $f(x)$ and we understand $\nabla{f}(x)$ as an element in $\partial{f}(x)$.
Since the H\"older condition implies the inequality
\begin{equation} 
f(x)-f(y)-\innprod{\nabla{f}(y)}{x-y}
\leq \frac{M}{1+\nu}\norm{x-y}^{1+\nu},\quad \forall x,y \in Q,
\end{equation}
defining $m_f(y;x)$ as (i) for $\sigma_f \in \sigma(f)$, it admits (\ref{s-conv-assump}) and (\ref{struc-ineq}) with $L(\cdot):=0$, $\bar\sigma_f:=\sigma_f$, and $\delta(\cdot,\cdot):=\frac{M}{1+\nu}\norm{x-y}^{1+\nu}$. The weakly smooth version of the composite and the saddle structures can also be considered in the same way.

For the weakly smooth problems, Nemirovski and Nesterov \cite{NN85} (see also \cite[Section 2.3]{Els93}) proposed an optimal PGM with the (optimal) complexity bounds
\begin{equation}
\label{lower-compl-weak}
c_1(\rho)\left(\frac{M}{\varepsilon}\right)^{\frac{2}{3\rho-2}}\left(\frac{d(x^*)}{\sigma_d}\right)^{\frac{\rho}{3\rho-2}} ~~~\textrm{and}~~~ c_2(\rho)\left(\frac{M^2}{\sigma^{\rho}}\frac{1}{\varepsilon^{2-\rho}}\right)^{\frac{1}{3\rho-2}},
\end{equation}
for non-strongly and strongly convex cases, respectively, where $\rho:=1+\nu \in [1,2)$, $c_1(\cdot),c_2(\cdot)$ are continuous functions, and $\sigma>0$ is a convexity parameter of $f$ with respect to the norm $\norm{\cdot}$; the proposed method is further applicable for more general classes of problems.  Moreover, Nesterov \cite{Nes15u} improved a restriction of the method in the non-strongly convex case in the sense that the proposed method ensures the optimal convergence rate without fixing the iteration number. It is important to note that the methods proposed by \cite{NN85} and \cite{Nes15u} can achieve the above complexity of iterations for non-strongly convex case even if we do not know $M$ and $\nu$ while the proposed method here needs an additional (but relatively small) ``cost" for estimating $M$. This approach can be also seen in \cite{BT09,Nes83,Nes13} for an estimation of the Lipschitz constant $M$ in the case $\nu=1$.
The studies \cite{DGNs,DGN} of the inexact oracle model are also important; they proposed an optimal method for weakly smooth problems in the non-strongly convex case and a sub-optimal one in the strongly convex case (PGMs for uniformly convex functions are also discussed).

A convergence result for CGMs for this class can be also obtained in the same way as the smooth problems which ensures the complexity $O((MR/\varepsilon)^{1/\nu})$ where $R:={\rm Diam}(Q)$ (see \cite[Proposition 1.1]{CJN13} and \cite{Nes15}).

\item
The objective functions in (i) and (iv) can be simultaneously considered by assuming
\[ f(y)-f(x)-\innprod{g(y)}{y-x} \leq \frac{L}{2}\norm{y-x}^2+\frac{M}{\rho}\norm{y-x}^\rho,\quad \forall x,y \in Q, \]
for a subgradient mapping $g(x) \in \partial{f}(x)$, $L,M \geq 0$, and $\rho \in [1,2)$. When $\sigma_f \in \sigma(f)$, we can take $m_f(y;x):=f(y)+\innprod{g(y)}{x-y}+\sigma_f\xi(y,x)$ to obtain (\ref{s-conv-assump}) and (\ref{struc-ineq}) with $L(\cdot):=L,~\bar\sigma_f:=\sigma_f$, and $\delta(y,x):=\frac{M}{\rho}\norm{y-x}^{\rho}$.
When $\sigma_f=0$ or the prox-function grows quadratically on $Q$, (nearly) optimal PGMs for this model in the case $\rho=1$ are studied in \cite{CLP12,GL12,GL13,Lan12,Lan14b} with a stochastic setting. 

\end{enumerate}


\subsubsection{Existing methods for structured problems}
\label{sssec-existing-struc}

We finally describe some particular PGMs and CGMs which will be important for the comparison with the proposed methods in the paper.
For that, we introduce two kinds of update formulas of gradient-based methods.

The first is the Classical Gradient Method \cite[Method 16]{IF14}, which performs as follows: For given weight $\{\lambda_k\}_{k \geq 0}$ and scaling parameters $\{\beta_k\}_{k \geq -1}$, generate $\{\sz_k\}_{k \geq -1}$ and $\{x_k\}_{k \geq 0}$ by the update (\ref{existing-SGM}) with the model (\ref{eMD-subprob-nonst}) or (\ref{DA-subprob-nonst}), and set $\{\hat{x}_k\}_{k \geq 0}$ by $\hat{x}_k=\sum_{i=0}^k\lambda_ix_i/\sum_{i=0}^k\lambda_i$.
The {\it primal} and {\it dual gradient methods} in \cite{Nes13} for the composite problems (ii) and in \cite{DGN} for the inexact oracle model (iii) are closely related to this algorithm in the non-strongly convex case. A further relation in the strongly convex case will be presented in this paper.

The second, the Fast Gradient Method (FGM) \cite[Method 17]{IF14}, is described as follows: For given weight $\{\lambda_k\}_{k \geq 0}$ and scaling parameters $\{\beta_k\}_{k \geq -1}$, set $x_0:=\z_{-1}:=\argmin_{x \in Q}d(x)$, $\hat{x}_0:=\z_0$ and, for $k \geq 0$, iterate
\begin{equation}\label{FGM-single}
\begin{array}{rcl}
x_{k+1}&:=&(1-\tau_k)\hat{x}_k+\tau_k\z_k,\quad \textrm{where } \tau_k:=\frac{\lambda_{k+1}}{\sum_{i=0}^{k+1}\lambda_i},\\
\hat{x}_{k+1}&:=&(1-\tau_k)\hat{x}_k+\tau_k\z_{k+1},
\end{array}
\end{equation}
where $\z_k$ is determined by the fixed subproblem either the extended MD model (\ref{eMD-subprob-nonst}) or the DA model (\ref{DA-subprob-nonst}).
It was indicated in \cite{IF14} that the FGM with $\lambda_0:=1,\lambda_{k+1}:=\frac{1+\sqrt{1+4\lambda_{k}^2}}{2}~(k \geq 0)$, and $\beta_k\equiv L/\sigma_d$ yields Tseng's accelerated PGMs \cite{Tseng10} for the composite problems which achieve the convergence rate $f(\hat{x}_k)-f(x^*) \leq O(Ld(x^*)/(\sigma_d k^2))$ yielding the optimal complexity $O(\sqrt{Ld(x^*)/(\sigma_d\varepsilon)})$ as (i) in the non-strongly convex case.

Furthermore, the algorithm (\ref{FGM-single}) is also closely related to the following PGM and CGM, which will be unified in the framework of this paper:
\begin{itemize}
\item
Replacing the second update in (\ref{FGM-single}) by $\hat{x}_{k+1}:=(1-\tau_k)\hat{x}_k+\tau_k\w_{k+1}$, determining $\w_k$ and $\z_k$ by (\ref{eMD-subprob-nonst}) and (\ref{DA-subprob-nonst}) with $\beta_k:=L/\sigma_d$, respectively, the corresponding method with $\lambda_k:=(k+1)/2$ yields the Nesterov's optimal PGM \cite[Section 5.3]{Nes05s} for the smooth problems in the non-strongly convex case.
We remark that the achievement of the optimal complexity of the FGM and this Nesterov's PGM in the strongly convex case are not known without using multistage procedure; in the Euclidean setting, it turns out that a multistage procedure for them attains the optimal complexity $O(\sqrt{L/\sigma_f}\log(1/\varepsilon))$ in the strongly convex case (see, \eg, \cite[Section 5.1]{Nes13})%
\footnote{In fact, since they have the convergence rate $f(\hat{x}_k)-f(x^*)\leq \frac{cL\norm{x_0-x^*}_2^2}{2k^2}$ for a constant $c>0$, after $k\geq \sqrt{2cL/\sigma_f}$ iterations, we have $f(\hat{x}_k)-f(x^*)\leq \frac{\sigma_f}{4}\norm{x_0-x^*}_2^2\leq \frac{1}{2}(f(x_0)-f(x^*))$ by the strong convexity of $f$ and the optimality of $x^*$. Then repeating $O(\log_2(1/\varepsilon))$ times of restarting the method every $\sqrt{2cL/\sigma_f}$ iterations, it ensures an $\varepsilon$-solution.}.
\item Letting $\lambda_k:=(k+1)/2$ and assuming the boundedness of $Q$, the algorithm (\ref{FGM-single}) with the subproblems (\ref{eMD-subprob-nonst}) and (\ref{DA-subprob-nonst}) with $\beta_k\equiv 0$ corresponds to Lan's modified CGMs, Algorithms 4 and 5, respectively, in \cite{Lan14a} with the stepsize policy $\alpha_k:=2/(k+1)$ and $\theta_k:=k$.

\end{itemize}

On the other hand, the classical CGM \cite{DR70,FW56,PD78} for smooth problems is basically performed as follows: Choose $x_0 \in Q$ and, for $k \geq 0$, iterate 
\begin{equation}
\label{classical-CGM}
\sz_{k} \in \Argmin_{x \in Q}\innprod{\nabla{f}(x_k)}{x-x_k},\quad x_{k+1}:=(1-\tau_k)x_k+\tau_k\sz_{k},\quad  k\geq 0
\end{equation}
where $\tau_k \in [0,1)$ (we assume the boundedness of $Q$). Excepting the Lan's modified CGMs, all the above mentioned CGMs are based on this classical CGM.
Notice that the subproblem can be seen as the extended MD model (\ref{eMD-subprob-nonst}) with $\beta_k \equiv 0$.


\section{Unifying framework for (sub)gradient-based methods}
\label{sec-framework}

In this section we define the {\it unifying framework}, namely Methods \ref{gen-alg-nonsm} and \ref{gen-alg-struc} combined with Property \ref{framework-single} and \ref{framework-double}, which provides a generalization of some existing methods and new convergence results with a universal analysis.
The proposed methods require the computation of minimizer(s) $\z_k$ (and $\w_k$) of one or two auxiliary problem(s) at each iterations as the existing methods presented in Sections \ref{ssec-existing-nonsm} and \ref{sssec-existing-struc}.
In order to simplify the notation, we introduce {\it auxiliary functions} $\sphi_k(x)$ and $\apsi_k(x)$, and denote the minimizers of our subproblems as $\z_k:=\argmin_{x \in Q}\aphi_k(x)$ and $\w_k:=\argmin_{x \in Q}\apsi_k(x)$.

Now let us see how we proceed in specifying our (sub)gradient-based methods.
They are determined by the parameters $\{\lambda_k\}_{k \geq 0}$, $\{\beta_k\}_{k \geq -1}$, and functions $\{(\sphi_k(x),\apsi_k(x))\}_{k \geq -1}$, where
\begin{itemize}
\item $\{\lambda_k\}_{k\geq 0}$ is a sequence of positive real numbers, the {\it weight parameters},
\item $\{\beta_k\}_{k \geq -1}$ is a nondecreasing sequence of nonnegative real numbers, the {\it scaling parameters}, and
\item $\{(\sphi_k(x),\apsi_k(x))\}_{k \geq -1}$ is a {\it coupled sequence of auxiliary functions} which are minimized at each iterations.
\end{itemize}
We always assume that weight parameters are positive and that scaling parameters are nonnegative and nondecreasing.
Remark that these objects are possibly determined in a recursive manner during the methods.
Then our methods generate the following sequences in $Q$.
\begin{itemize}

\item $\{x_k\}_{k \geq 0}$ is the sequence of test points for which we evaluate $m_f(x_k;x)$.

\item $\{\z_k\}_{k \geq -1}$ is the sequence of solutions of subproblems $\min_{x \in Q}\aphi_k(x)$.

\item $\{\w_k\}_{k \geq -1}$  is the sequence of solutions of subproblems $\min_{x \in Q}\apsi_k(x)$.

\item $\{\hat{x}_k\}_{k \geq 0}$ is the sequence of approximate solutions for the problem (\ref{primal-problem}).
\end{itemize}
In view of our actual construction defined in Section~\ref{ssec-auxfunc-construc}, we suppose that the auxiliary functions $\{(\aphi_k(x),\apsi_k(x))\}_{k\geq -1}$ are constructed associated with weight parameters $\{\lambda_k\}_{k \geq 0}$, scaling parameters $\{\beta_k\}_{k \geq -1}$, and test points $\{x_k\}_{k \geq 0}$ in a recursive manner.
We often consider the case of a single sequence $\{\sphi_k(x)\}_{k \geq -1}$ of auxiliary functions which can be regarded as the case $\apsi_k(x) \equiv \aphi_k(x)$.

We will gradually specify the above general objects by giving explicit update formulas in three steps: The first is for the points $\{x_k\}_{k \geq 0}$ and $\{\hat{x}_k\}_{k \geq 0}$ by proposing general (sub)gradient-based methods (Section \ref{ssec-gen-subgrad-methods}), the second is for the auxiliary functions $\{(\aphi_k(x),\apsi_k(x))\}_{k \geq -1}$ used in the general methods (Section \ref{ssec-auxfunc-construc}), and the final is for the parameters $\{\lambda_k\}_{k \geq 0}$ and $\{\beta_k\}_{k \geq -1}$ to provide efficient convergences (Section \ref{sec-convergence}).


\subsection{General properties for the construction of auxiliary functions in the unifying framework}

We begin by describing general properties which the auxiliary functions $\{(\aphi_k(x),\apsi_k(x))\}_{k \geq -1}$ should satisfy.
These properties will guide us in how to iteratively construct the auxiliary functions.
The first set of properties is for a sequence of auxiliary functions $\{\sphi_k(x)\}_{k \geq -1}$. We define $\sum_{i=0}^{-1}(\cdot):=0$ and so $S_{-1}=0$.

\begin{prope}[in the unifying framework]
\label{framework-single}
Let $\{\sphi_k(x)\}_{k \geq -1}$ be a sequence of auxiliary functions associated with weight parameters $\{\lambda_k\}_{k \geq 0}$, scaling parameters $\{\beta_k\}_{k \geq -1}$, and test points $\{x_k\}_{k \geq 0}$. Let $\sigma_f\geq 0$ be a convexity parameter satisfying (\ref{s-conv-assump}) for some lower approximation model $m_f(y;x)$ of $f(x)$. Denote $\sz_k:=\argmin_{x \in Q}\sphi_k(x)$\footnote{The auxiliary function $\sphi_k(x)$ can possibly be an affine function. In that case, we will assume the boundedness of $Q$ in order to ensure an existence of a minimizer $\sz_k$.}. Then, the following conditions hold:
\begin{itemize}
\leftskip=3truemm
\item[(A1)] $\sphi_{-1}(\sz_{-1}) = 0$ and $\sz_{-1}=x_0$.
\item[(A2)] $\forall k \geq -1,~ \forall x \in Q$, we have
\[\sphi_{k+1}(x) \geq \sphi_k(\sz_k) + \lambda_{k+1}m_f(x_{k+1};x)+ \beta_{k+1}d(x) - \beta_k l_d(\sz_k;x) + S_k\sigma_f\xi(\sz_k,x).\]
\item[(A3)] $\forall k \geq -1,~~ \sphi_k(\sz_k) \leq \min_{x \in Q}\left\{ \sum_{i=0}^k \lambda_im_f(x_i;x) + \beta_k l_d(\sz_k;x) - S_k \sigma_f\xi(\sz_k,x) \right\}$.
\end{itemize}
\end{prope}

The above property is a generalization of Property 2 \cite{IF14} which is particularized by taking $\sigma_f=0$.
As a simple extension of Property \ref{framework-single}, we further consider a coupled sequence $\{(\aphi_k(x),\apsi_k(x))\}_{k \geq -1}$ of auxiliary functions which admits the property below.

\begin{prope}[in the unifying framework]
\label{framework-double}
Let $\{(\aphi_k(x),\apsi_k(x))\}_{k \geq -1}$ be a coupled sequence of auxiliary functions associated with weight parameters $\{\lambda_k\}_{k \geq 0}$, scaling parameters $\{\beta_k\}_{k \geq -1}$, and test points $\{x_k\}_{k \geq 0}$. Denote $\z_k:=\argmin_{x \in Q}\aphi_k(x)$ and $\w_k:=\argmin_{x \in Q}\apsi_k(x)$. Let $\sigma_f\geq 0$ be a convexity parameter satisfying (\ref{s-conv-assump}) for some lower approximation model $m_f(y;x)$ of $f(x)$. Then, the following conditions hold:
\begin{itemize}
\leftskip=3truemm
\item[(B0)] $\aphi_k(x) \geq \apsi_k(x)$ for all $x \in Q$.
\item[(B1)] $\apsi_{-1}(\w_{-1}) = 0$ and $\z_{-1}=\w_{-1}=x_0$.
\item[(B2)] $\forall k \geq -1,~ \forall x \in Q$, we have
\[\apsi_{k+1}(x) \geq \aphi_k(\z_k) + \lambda_{k+1}m_f(x_{k+1};x)+ \beta_{k+1}d(x) - \beta_k l_d(\z_k;x) + S_k\sigma_f\xi(\z_k,x).\]
\item[(B3)] $\forall k \geq -1,~~ \apsi_k(\w_k) \leq \min_{x \in Q}\left\{ \sum_{i=0}^k \lambda_im_f(x_i;x) + \beta_kl_d(\z_k;x) - S_k \sigma_f\xi(\z_k,x) \right\}$.
\end{itemize}
\end{prope}

Note that letting $\apsi_k(x) \equiv \aphi_k(x)$, it yields Property \ref{framework-single}.


\subsection{(Sub)gradient-based methods in the unifying framework}
\label{ssec-gen-subgrad-methods}

We propose the following (sub)gradient-based methods for non-smooth problems (Method \ref{gen-alg-nonsm}) and structured problems (Method \ref{gen-alg-struc}), respectively.
Each of them have two types of updates, the classical and the modified ones.

\begin{algorithm}[Subgradient-based methods for non-smooth problems in the unifying framework]
\label{gen-alg-nonsm}
Consider a non-smooth problem in the class $\mathcal{NSP}(g, \sigma_f)$.
Let $\{\lambda_k\}_{k \geq 0}$ and $\{\beta_k\}_{k \geq -1}$ be sequences of weight and scaling parameters, respectively.
Generate a sequence $\{(\sz_{k-1},x_k,g_k,\hat{x}_k)\}_{k \geq 0}$ by either the classical or the modified method as follows.
\begin{itemize}
\item[(0)] Set $\hat{x}_0 := x_0 := \sz_{-1} := \argmin_{x \in Q}d(x)$.
\item[(1)] ($k$-th iteration, $k \geq 0$) Set
$g_k:=g(x_k) \in \partial{f}(x_k)$ and compute $\sz_k, x_{k+1}, \hat{x}_{k+1}$ by
\begin{eqnarray*}
\textrm{Classical method} &:& x_{k+1}:=\sz_k:=\argmin_{x \in Q}\sphi_k(x),\quad \hat{x}_{k+1}:= \frac{S_k\hat{x}_k+\lambda_{k+1}\sz_k}{S_{k+1}},\\
or\\
\textrm{Modified method} &:& \sz_k := \argmin_{x \in Q}\sphi_k(x),\quad \hat{x}_{k+1}:=x_{k+1}:= \frac{S_k\hat{x}_k+\lambda_{k+1}\sz_k}{S_{k+1}}.
\end{eqnarray*}
\end{itemize}
Here, $\{\sphi_k(x)\}_{k \geq -1}$ is a single sequence of auxiliary functions satisfying Property \ref{framework-single}.
\end{algorithm}

Note that we did not use a coupled sequence $\{(\aphi_k(x),\apsi_k(x))\}_{k \geq -1}$ of auxiliary functions because we will see that their analysis (Lemmas \ref{rel-0}, \ref{rel-k-clas}, and \ref{rel-k-modif}) for the non-smooth problems are independent of the second object $\{\apsi_k(x)\}_{k \geq -1}$ (or $\w_k$).

\begin{algorithm}[Gradient-based methods for structured problems in the unifying framework]\label{gen-alg-struc}
Consider a structured problem in the class $\mathcal{SP}(m_f,\sigma_f,\bar\sigma_f,L,\delta)$.
Let $\{\lambda_k\}_{k \geq 0}$ and $\{\beta_k\}_{k \geq -1}$ be sequences of weight and scaling parameters, respectively.
Generate a sequence $\{(\z_{k-1},\w_{k-1},x_k,\hat{x}_k)\}_{k \geq 0}$ by either the classical or the modified method as follows.

\begin{itemize}
\item[(0)] Set $x_0:=\z_{-1}:=\w_{-1}:=\argmin_{x \in Q}d(x)$. Compute
\[ \z_0:=\argmin_{x \in Q}\aphi_0(x),\quad \hat{x}_0 := \w_0 := \argmin_{x \in Q}\apsi_0(x). \]

\item[(1)]($k$-th iteration, $k \geq 0$) Set
\begin{eqnarray*}
x_{k+1}&:=&\left\{
	\begin{array}{lcl}
	\z_k&:&\textrm{Classical method, }\\
	\dfrac{S_k\hat{x}_k+\lambda_{k+1}\z_k}{S_{k+1}}&:&\textrm{Modified method, }
	\end{array}
\right.\\
\z_{k+1} &:=& \argmin_{x \in Q}\aphi_{k+1}(x),\\
\w_{k+1} &:=& \argmin_{x \in Q}\apsi_{k+1}(x),\\
\hat{x}_{k+1} &:=& \frac{S_k\hat{x}_k + \lambda_{k+1}\w_{k+1}}{S_{k+1}}.
\end{eqnarray*}
\end{itemize}
Here, $\{(\aphi_k(x),\apsi_k(x))\}_{k \geq -1}$ is a coupled sequence of auxiliary functions satisfying Property \ref{framework-double}.
\end{algorithm}

The implementation of these methods will require a more specific construction of auxiliary functions $\{(\aphi_k(x),\apsi_k(x))\}_{k \geq -1}$ as we will see next.


\subsection{Construction of auxiliary functions in the unifying framework}
\label{ssec-auxfunc-construc}

Here we provide some formulas to construct a coupled sequence $\{(\aphi_k(x),\apsi_k(x))\}_{k \geq -1}$ of auxiliary functions which admit Property \ref{framework-double}.
For that, we firstly construct a single sequence of auxiliary functions $\{\sphi_k(x)\}_{k \geq -1}$ satisfying Property \ref{framework-single}.

\begin{thm}\label{const-auxfunc-single}
Given the weight parameters $\{\lambda_k\}_{k \geq 0}$, the scaling parameters $\{\beta_k\}_{k \geq -1}$, the test points $\{x_k\}_{k \geq 0}$, and a convexity parameter $\sigma_f \geq 0$ satisfying (\ref{s-conv-assump}) for some lower approximation model $m_f(y;x)$ of $f(x)$, construct the sequence $\{\sphi_k(x)\}_{k \geq -1}$ of auxiliary functions as follows. $\sphi_{-1}(x) := \beta_{-1}d(x)$, $\sz_{-1}:=x_0$ and,
for $k \geq -1$, define
\begin{equation}\label{eMD-update}
\sphi_{k+1}(x) := \sphi_k(\sz_k) + \lambda_{k+1}m_f(x_{k+1};x) + \beta_{k+1}d(x) - \beta_k l_d(\sz_k;x) + S_k\sigma_f\xi(\sz_k,x)
\end{equation}
or
\begin{equation}\label{DA-update}
\sphi_{k+1}(x) := \sphi_k(x) + \lambda_{k+1}m_f(x_{k+1};x) + \beta_{k+1}d(x) - \beta_{k}d(x).
\end{equation}
Then, the sequence $\{\sphi_k(x)\}_{k\geq -1}$ satisfies Property \ref{framework-single}.
\end{thm}

The assumption $\sz_{-1}:=x_0$ is satisfied whenever $\beta_{-1}>0$ because $\min_{x \in Q}d(x)=d(x_0)=0$, but it is required when $\beta_{-1}=0$; in both cases, the condition (A1) holds.
To prove Theorem \ref{const-auxfunc-single}, it remains to show (A2) and (A3) which will be done in Lemmas \ref{admits-2} and \ref{admits-3}, respectively.

The following theorem is a simple consequence of Theorem \ref{const-auxfunc-single}.

\begin{thm}\label{const-auxfunc-double}
Let $\{\aphi_k(x)\}_{k \geq -1}$ be generated accordingly to the construction in Theorem \ref{const-auxfunc-single} associated with weight parameters $\{\lambda_k\}_{k \geq 0}$, scaling parameters $\{\beta_k\}_{k \geq -1}$, test points $\{x_k\}_{k \geq 0}$, and a convexity parameter $\sigma_f \geq 0$ satisfying (\ref{s-conv-assump}) for some lower approximation model $m_f(y;x)$ of $f(x)$. Define $\{\apsi_k(x)\}_{k \geq -1}$ by $\apsi_{-1}(x):=\aphi_{-1}(x)$ and
\begin{equation}\label{double-update}
\apsi_{k+1}(x):=\aphi_k(\z_k)+\lambda_{k+1}m_f(x_{k+1};x)+\beta_{k+1}d(x)-\beta_kl_d(\z_k;x)+S_k\sigma_f\xi(\z_k,x).
\end{equation}
Then, the sequence $\{(\aphi_k(x),\apsi_k(x))\}_{k\geq -1}$ satisfies Property \ref{framework-double}.
\end{thm}

\begin{proof}
Notice that (\ref{double-update}) satisfies the condition (B2) as equality. The condition (B1) is immediate from the condition (A1) for $\{\aphi_k(x)\}$ and the definition $\apsi_{-1}(x):=\aphi_{-1}(x)$.
Since (\ref{double-update}) coincides with the right hand side of (A2) for $\{\aphi_k(x)\}$, the condition (B0) is clear.
Finally, the condition (B3) is satisfied by (B0) and (A3) for $\{\aphi_k(x)\}$.
\end{proof}

Before proving Theorem \ref{const-auxfunc-single}, let us see some particular constructions of auxiliary functions, which will be useful for the comparison with some existing methods.
\begin{itemize}

\item {\it Extended MD model}. Define $\{\aphi_k(x)\}_{k \geq -1}$ by $\aphi_{-1}(x):=\beta_{-1}d(x)$ and
\begin{equation}\label{eMD-model}
\aphi_{k+1}(x):=\aphi_k(\z_k)+\lambda_{k+1}m_f(x_{k+1};x)+\beta_{k+1}d(x)-\beta_kl_d(\z_k;x)+S_k\sigma_f\xi(\z_k,x)
\end{equation}
for $ k \geq -1$. Then, Property \ref{framework-single} follows from Theorem \ref{const-auxfunc-single} with the update (\ref{eMD-update}).

\item {\it DA model}. Define $\{\aphi_k(x)\}_{k \geq -1}$ by $\sphi_{-1}(x):=\beta_{-1}d(x)$ and
\begin{equation}\label{DA-model}
\aphi_{k}(x):=\sum_{i=0}^k\lambda_im_f(x_i;x) + \beta_k d(x)
\end{equation}
for $k \geq -1$. Then, Property \ref{framework-single} follows from Theorem \ref{const-auxfunc-single} with the update (\ref{DA-update}).

\item {\it Hybrid model}. Define $\{(\aphi_k(x),\apsi_k(x))\}$ by $\apsi_{-1}(x):=\beta_{-1}d(x)$ and
\begin{equation}\label{hybrid-model}
\begin{array}{rcl}
\aphi_k(x) &:=& \sum_{i=0}^k \lambda_im_f(x_i;x)+\beta_k d(x),\\
\apsi_{k+1}(x) &:=& \min_{z\in Q}\aphi_k(z)+\lambda_{k+1}m_f(x_{k+1};x)+\beta_{k+1}d(x)-\beta_kl_d(\z_k;x)+S_k\sigma_f\xi(\z_k,x)
\end{array}
\end{equation}
for $k \geq -1$. Then, Property \ref{framework-double} follows from Theorem \ref{const-auxfunc-double} with the update (\ref{double-update}).
\end{itemize}

Consequently, Method \ref{gen-alg-nonsm} provides at least four particularizations; we can choose the classical or the modified updates combined to the choice of the auxiliary functions constructed by the extended MD model (\ref{eMD-model}) or by the DA model (\ref{DA-model}) (or arbitrarily combination of them). Notice that subproblems $\sz_k := \argmin_{x \in Q}\sphi_k(x)$ in these particularizations can be solved as the form (\ref{subprob-simple}).

 Method \ref{gen-alg-struc} yields at least six particularizations due to the additional choice of the hybrid model (\ref{hybrid-model}). However, employing the models (\ref{eMD-model}) or (\ref{DA-model}) in Method \ref{gen-alg-struc} reduces the number of subproblems at each iteration since $\z_k\equiv \w_k$. Note that only the extended MD model (\ref{eMD-model}) turns the subproblem $\z_k:=\argmin_{x \in Q}\aphi_k(x)$ of the form (\ref{subprob-struc}); the others require the solution of the subproblem (\ref{DA-subprob-nonst}). However, the subproblems with these models have the same difficulty for all the examples cited in Section \ref{ssec-existing-struc}.

We remark that Theorems \ref{const-auxfunc-single} and \ref{const-auxfunc-double} give infinitely many ways of constructing$\{(\aphi_k(x),\apsi_k(x))\}$ because we can mix the updates (\ref{eMD-update}) and (\ref{DA-update}) in any order.


\subsection{Proof of Theorem \ref{const-auxfunc-single}}
Now let us complete the proof of Theorem \ref{const-auxfunc-single}.

\begin{lem}\label{prop-subprob}
Let $\{\sphi_k(x)\}_{k \geq -1}$ be generated accordingly to the construction in Theorem \ref{const-auxfunc-single} associated with weight parameters $\{\lambda_k\}_{k \geq 0}$, scaling parameters $\{\beta_k\}_{k \geq -1}$, test points $\{x_k\}_{k \geq 0}$, and a convexity parameter $\sigma_f \geq 0$ satisfying (\ref{s-conv-assump}) for some lower approximation model $m_f(y;x)$ of $f(x)$.
Then, for every $k \geq -1$, we have
\[ \sphi_k(x) \geq \sphi_k(\sz_k)+(\beta_k+S_k\sigma_f)\xi(\sz_k,x),\quad \forall x \in Q,~\forall k \geq -1. \]
\end{lem}

\begin{proof}
Since $\sigma_f \in \sigma(m_f(x_i,\cdot))$ for $i \geq 0$, we can see inductively that $\beta_{k}+S_{k}\sigma_f \in \sigma(\sphi_{k})$ for all $k \geq -1$. Therefore, using its characterization (\ref{st-conv-Breg}), the optimality condition $\sphi_k'(\sz_k;x-\sz_k) \geq 0, \forall x \in Q$ for the minimizer $\sz_k=\argmin_{x \in Q}\sphi_k(x)$ yields the conclusion.
\end{proof}

\begin{lem}\label{admits-2}
Let $\{\sphi_k(x)\}_{k \geq -1}$ be generated accordingly to the construction in Theorem \ref{const-auxfunc-single} associated with weight parameters $\{\lambda_k\}_{k \geq 0}$, scaling parameters $\{\beta_k\}_{k \geq -1}$, test points $\{x_k\}_{k \geq 0}$, and a convexity parameter $\sigma_f \geq 0$ satisfying (\ref{s-conv-assump}) for some lower approximation model $m_f(y;x)$ of $f(x)$.
Then, the condition (A2) holds.
\end{lem}

\begin{proof}
Notice that the construction (\ref{eMD-update}) satisfies (A2) as equality.
In the case of the construction (\ref{DA-update}), Lemma \ref{prop-subprob} yields for any $x \in Q$ that
\begin{eqnarray*}
\sphi_{k+1}(x) &=&  \sphi_k(x) + \lambda_{k+1}m_f(x_{k+1};x) + \beta_{k+1}d(x) - \beta_{k}d(x) \\
	&\geq& [\sphi_k(\sz_k) + (\beta_k+S_k\sigma_f)\xi(\sz_k,x) ] + \lambda_{k+1}m_f(x_{k+1};x) + \beta_{k+1}d(x) - \beta_{k}d(x)\\
	&=& \sphi_k(\sz_k) + \lambda_{k+1}m_f(x_{k+1};x) + \beta_{k+1}d(x) - \beta_{k}l_d(\sz_k;x) + S_k\sigma_f\xi(\sz_k,x)
\end{eqnarray*}
which is the condition (A2) for $k \geq -1$.
\end{proof}

\begin{lem}\label{admits-3}
Let $\{\sphi_k(x)\}_{k \geq -1}$ be generated accordingly to the construction in Theorem \ref{const-auxfunc-single} associated with weight parameters $\{\lambda_k\}_{k \geq 0}$, scaling parameters $\{\beta_k\}_{k \geq -1}$, test points $\{x_k\}_{k \geq 0}$, and a convexity parameter $\sigma_f \geq 0$ satisfying (\ref{s-conv-assump}) for some lower approximation model $m_f(y;x)$ of $f(x)$.
Then, the condition (A3) holds.
\end{lem}

\begin{proof}
We prove the assertion by induction. Since $\sz_{-1}=x_0=\argmin_{x \in Q}d(x)$, we have $\min_{x \in Q}l_d(\sz_{-1};x)=\min_{x \in Q}d(x)=0$ which proves (A3) for $k=-1$.
Assume that (A3) holds up to $k \geq -1$.
In the case when all $\{\aphi_{i}(x)\}_{i=0}^{k+1}$ are constructed by (\ref{DA-update}), it coincides with the formula (\ref{DA-model}). Therefore, Lemma \ref{prop-subprob} implies that
\begin{eqnarray*}
\sphi_k(\sz_k) \leq \sphi_k(x) - (\beta_k+S_k\sigma_f)\xi(\sz_k,x)
&=& \sum_{i=0}^k \lambda_i m_f(x_i;x) + \beta_kd(x) - (\beta_k + S_k\sigma_f)\xi(\sz_k,x)\\
&=& \sum_{i=0}^k \lambda_i m_f(x_i;x) + \beta_kl_d(\sz_k;x) - S_k\sigma_f\xi(\sz_k,x)
\end{eqnarray*}
for every $x \in Q$, from which the condition (A3) follows.
If this is not the case, there exists some integer $j\leq k$ such that $\sphi_{k+1}(x)$ is constructed as defining $\sphi_{j+1}(x)$ by (\ref{eMD-update}) and $\sphi_{j+2}(x),\ldots,\sphi_{k+1}(x)$ by (\ref{DA-update}). Then, we have
\[
\sphi_{k+1}(x) =\min_{z\in Q}\sphi_{j}(z) + \sum_{i=j+1}^{k+1}\lambda_im_f(x_i;x) + \beta_{k+1}d(x)-\beta_jl_d(\sz_j;x) + S_j\sigma_f\xi(\sz_j,x)
\]
which yields $\sphi_{k+1}(x)\leq\sum_{i=0}^{k+1}\lambda_im_f(x_i;x) + \beta_{k+1}d(x)$ by the induction hypothesis (A3) for $\sphi_j(x)$.
Therefore, Lemma \ref{prop-subprob} implies for every $x \in Q$ that
\begin{eqnarray*}
\sphi_{k+1}(\sz_{k+1}) &\leq& \sphi_{k+1}(x)- (\beta_{k+1}+S_{k+1}\sigma_f)\xi(\sz_{k+1},x)\\
&\leq& \sum_{i=0}^{k+1} \lambda_im_f(x_i;x) + \beta_{k+1}d(x) - (\beta_{k+1}+S_{k+1}\sigma_f)\xi(\sz_{k+1},x)\\
&=& \sum_{i=0}^{k+1} \lambda_im_f(x_i;x) + \beta_{k+1}l_d(\sz_{k+1};x) - S_{k+1}\sigma_f\xi(\sz_{k+1},x)
\end{eqnarray*}
which gives the condition (A3) for $\sphi_{k+1}(x)$.
\end{proof}


\section{General convergence estimates of subgradient-based methods in the unifying framework}
\label{sec-general-convergence}

In this section we show general efficiency estimates of Methods \ref{gen-alg-nonsm} and \ref{gen-alg-struc} for the non-smooth and for the structured problems, respectively.
We then use the results of this section to derive particular convergence rates for these methods in Section \ref{sec-convergence}.

Note that in general the classical and the modified methods in Methods \ref{gen-alg-nonsm} and \ref{gen-alg-struc} will provide different convergence rates.
They yield the same convergence rate for non-smooth problems but the modified method gives much better efficiency than the classical method for smooth problems as discussed in Section \ref{sec-convergence}.

The following theorems show general estimates for Methods \ref{gen-alg-nonsm} and \ref{gen-alg-struc} which will be proved in the remainder of this section.


\begin{thm}
\label{gen-conv-nonsm}
Consider a non-smooth problem in the class $\mathcal{NSP}(g,\sigma_f)$.
Let $\{(\sz_{k-1},x_k,g_k,\hat{x}_k)\}_{k \geq 0}$ be generated by Method \ref{gen-alg-nonsm} associated with weight parameters $\{\lambda_k\}_{k \geq 0}$, scaling parameters $\{\beta_k\}_{k \geq -1}$. Then, for every $k \geq 0$, the estimate
\begin{equation}
\label{gen-est-nonsm}
f(\hat{x}_k) - f(x^*) + \sigma_f\xi(\sz_k,x^*) \leq \frac{\beta_kl_d(\sz_k;x^*)+C_k}{S_k}
\end{equation}
holds, where
\begin{equation}\label{C_k-nonsm}
C_k := \left\{ \begin{array}{lcl}
\frac{1}{2\sigma_d}\sum_{i=0}^k \frac{\lambda_i^2}{\beta_{i-1} + S_i\sigma_f}\norm{g_i}_*^2&&\textrm{for the classical method; and}\\
\frac{1}{2\sigma_d}\sum_{i=0}^k \frac{\lambda_i^2S_{i}}{\lambda_{i}^2\sigma_f + S_i(\beta_{i-1} + S_{i-1}\sigma_f)}\norm{g_i}_*^2 &&\textrm{for the modified method}.
\end{array} \right.
\end{equation}
Furthermore, for every $k \geq 0$, the above estimate holds even replacing the left hand side by $\frac{1}{S_k}\sum_{i=0}^k\lambda_if(x_i)-f(x^*) + \sigma_f\xi(\sz_k,x^*)$ or by $\min_{0\leq i\leq k}f(x_i)-f(x^*) + \sigma_f\xi(\sz_k,x^*)$ for the classical method.
\end{thm}


\begin{thm}\label{gen-conv-struc}
Consider a structured problem in the class $\mathcal{SP}(m_f,\sigma_f,\bar\sigma_f,L,\delta)$.
Let $\{(\z_{k-1},\w_{k-1},x_k,\hat{x}_k)\}_{k \geq 0}$ be generated by Method \ref{gen-alg-struc} associated with weight parameters $\{\lambda_k\}_{k \geq 0}$, scaling parameters $\{\beta_k\}_{k \geq -1}$. Then, for every $k \geq 0$, the estimate
\begin{equation}\label{gen-est-struc}
f(\hat{x}_k) - f(x^*) + \sigma_f\xi(\z_k,x^*) \leq \frac{\beta_kl_d(\z_k;x^*)+C_k}{S_k}
\end{equation}
holds, where
\begin{equation} 
C_k := \left\{
\begin{array}{l}
\frac{1}{2}\sum_{i=0}^k \lambda_i\left(L(x_i) - \sigma_d \left(\bar\sigma_f+ \frac{\beta_{i-1}+S_{i-1}\sigma_f}{\lambda_i} \right)\right)\norm{\w_i-x_i}^2 + \sum_{i=0}^k\lambda_i\delta(x_i,\w_i)\\
\hspace{8truecm}\textrm{for the classical method; and}\\
\frac{1}{2}\sum_{i=0}^k S_i\left(L(x_i)-\sigma_d\left(\bar\sigma_f + \frac{S_{i}(\beta_{i-1}+S_{i-1}\sigma_f)}{\lambda_i^2}\right)\right)\norm{\hat{x}_i-x_i}^2+\sum_{i=0}^k S_i\delta(x_i,\hat{x}_i)\\
\hspace{8truecm}\textrm{for the modified method.}
\end{array} \right.
\end{equation}
Furthermore, for every $k \geq 0$, the above estimate holds even replacing the left hand side by $\frac{1}{S_k}\sum_{i=0}^k\lambda_if(\w_{i})-f(x^*) + \sigma_f\xi(\z_k,x^*)$ or by $\min_{0\leq i\leq k}f(\w_{i})-f(x^*) + \sigma_f\xi(\z_k,x^*)$ for the classical method.
\end{thm}

\begin{rem}
\label{CGM-gen-bound}
Method \ref{gen-alg-struc} with $\sigma_f=\bar\sigma_f=0$ and $\beta_k\equiv 0$ yields several versions of CGMs because the constructed auxiliary functions are non-negative linear combinations of constants and $\{m_f(x_i;x)\}_{i=0}^k$. In this case, Theorem \ref{gen-conv-struc} implies that the modified method ensures
\begin{equation} 
f(\hat{x}_k)-f(x^*)\leq \frac{C_k}{S_k} \leq \frac{\frac{1}{2}{\rm Diam}(Q)^2\sum_{i=0}^kL(x_i)\frac{\lambda_i^2}{S_i}}{S_k}+\frac{\sum_{i=0}^kS_i\delta(x_i,\hat{x}_i)}{S_k}
\end{equation}
for all $k \geq 0$, because $\norm{\hat{x}_i-x_i}^2=\frac{\lambda_i^2}{S_i^2}\norm{\w_i-\z_{i-1}}^2 \leq \frac{\lambda_i^2}{S_i^2}{\rm Diam(Q)}^2$. Note that, if $m_f(y;\cdot)$ is affine for each $y \in Q$, then the classical CGM (\ref{classical-CGM}) with $\tau_k:=\lambda_{k+1}/S_{k+1}$ and $\hat{x}_k:=x_k$ also admits a similar estimate\footnote{The proof of \cite[Theorem 5.3]{FG14} replacing the notation $(h(\cdot),\lambda_{k+1},\tilde{\lambda}_{k+1},L_{k+1},\delta_{k+1},\bar{\alpha}_{k+1},\beta_{k+1},\alpha_k)$ of \cite{FG14} by $(-f(\cdot),x_k,\sz_k,L(x_k),\delta(x_k,x_{k+1}),\tau_k,S_k/\lambda_0,\lambda_k/\lambda_0)$ for $k \geq 0$ shows the desired estimate because showing the result uses the assumption \cite[eq.(52)]{FG14} with $(L,\delta)=(L_{k+1},\delta_{k+1})$ only at $(\lambda,\bar\lambda)=(\lambda_{k+2},\lambda_{k+1})$, which corresponds to our assumption (\ref{struc-ineq}) at $(x,y)=(x_k,x_{k+1})$.}
\begin{equation}\label{clas-CGM-bound}
f(x_k)-f(x^*) \leq \frac{\lambda_0[f(x_0)-m_f(x_0;\sz_0)]}{S_k}+\frac{\frac{1}{2}{\rm Diam}(Q)^2\sum_{i=1}^kL(x_{i-1})\frac{\lambda_i^2}{S_i}}{S_k}+\frac{\sum_{i=1}^kS_i\delta(x_{i-1},x_i)}{S_k}.
\end{equation}
\end{rem}


\subsection{Key strategy of the proof}

Under the assumptions of Theorems \ref{gen-conv-nonsm} or \ref{gen-conv-struc}, we will prove by induction that the relation
\[ (\rel_k) \quad S_kf(\hat{x}_k) \leq \apsi_k(\w_k) + C_k \]
holds for every $k \geq 0$, which is used to prove the estimates (\ref{gen-est-nonsm}) and (\ref{gen-est-struc}).
Furthermore, the relations 
\[ (\nrel_k) ~~ \sum_{i=0}^k \lambda_if(x_i) \leq \apsi_k(\w_k) + C_k \quad \textrm{ and } \quad (\srel_k) ~~ \sum_{i=0}^k \lambda_if(\w_i) \leq \apsi_k(\w_k) + C_k\]
are also useful to prove the latter assertion of Theorems \ref{gen-conv-nonsm} and \ref{gen-conv-struc}, respectively.

These relations yield the following estimate.

\begin{lem}\label{R_k-bound}
Suppose that a sequence $\{\hat{x}_k\}_{k \geq 0} \subset Q$ satisfies the relation $(\rel_k)$ for a coupled sequence $\{(\aphi_k(x),\apsi_k(x))\}_{k \geq -1}$ of auxiliary functions associated with weight parameters $\{\lambda_k\}_{k \geq 0}$, scaling parameters $\{\beta_k\}_{k \geq -1}$, and test points $\{x_k\}_{k \geq 0}$.
If the condition (B3) in Property \ref{framework-double} holds for a convexity parameter $\sigma_f\geq 0$ and for some lower approximation model $m_f(y;x)$ of $f(x)$,
then we have
\begin{equation}\label{general-convergence}
f(\hat{x}_k) - f(x) + \sigma_f\xi(\z_k,x) \leq \frac{\beta_k l_d(\z_k;x) + C_k}{S_k},\quad \forall x \in Q.
\end{equation}
\end{lem}

\begin{proof}
The assertion follows from the condition (B3) and the relation $(\rel_k)$; for any $x \in Q$, we have
\[ S_kf(\hat{x}_k) \leq \sum_{i=0}^k \lambda_im_f(x_i;x) + \beta_k l_d(\z_k;x) - S_k\sigma_f\xi(\z_k,x) +C_k \leq S_kf(x) + \beta_k l_d(\z_k;x) - S_k\sigma_f\xi(\z_k,x)+C_k. \]
\end{proof}

\begin{rem}\label{remark-R_k}
(1) Analogues of Lemma \ref{R_k-bound} easily show that $(\nrel_k)$ and (B3) imply the inequality
\[ \min_{0\leq i \leq k}f(x_i)-f(x) + \sigma_f\xi(\z_k,x) \leq \frac{1}{S_k}\sum_{i=0}^k\lambda_if(x_i)-f(x)+ \sigma_f\xi(\z_k,x) \leq \frac{\beta_kl_d(\z_k;x)+C_k}{S_k} \]
for $x \in Q$. The conditions $(\srel_k)$ and (B3) also conclude the same replacing $x_i$ by $\w_i$.
\\
(2) When $\sigma_f>0$, (\ref{general-convergence}) provides bounds for the distances to $x^*$ from $\hat{x}_k$ and $\z_k$: According to the facts $f(x)-f(x^*)\geq \sigma_f\xi(x^*,x)$ and $\xi(x,y)\geq \frac{\sigma_d}{2}\norm{x-y}^2$ for $x,y \in Q$, the bound (\ref{general-convergence}) implies
\[ \min\{\norm{\hat{x}_k-x^*}^2,\norm{\z_k-x^*}^2\} \leq \frac{1}{2}\norm{\hat{x}_k-x^*}^2+\frac{1}{2}\norm{\z_k-x^*}^2 \leq \frac{\beta_kl_d(\z_k;x^*)+C_k}{\sigma_f\sigma_d S_k}. \]
\end{rem}

Lemma \ref{R_k-bound} and Remark \ref{remark-R_k} (1) shows that, in order to complete Theorems \ref{gen-conv-nonsm} and \ref{gen-conv-struc}, it suffices to prove $(\rel_k)$ and its variants $(\nrel_k)$ or $(\srel_k)$. We now turn to the inductive proof of them.


\subsection{Validity of $(\rel_k),~(\nrel_k),$ and $(\srel_k)$ when $k=0$}

We start the proof of the case $k=0$ for our induction.
Note that the assumptions of (i) and (ii) in the following lemma are exactly the situations of the initialization step (0) in Methods \ref{gen-alg-nonsm} and \ref{gen-alg-struc}, respectively.

\begin{lem}\label{rel-0}
%
(i)
Consider a non-smooth problem in the class $\mathcal{NSP}(g,\sigma_f)$ and let $\{(\aphi_k(x),\apsi_k(x))\}_{k\geq -1}$ be a coupled sequence of auxiliary functions satisfying Property \ref{framework-double} associated with weight parameters $\{\lambda_k\}_{k \geq 0}$, scaling parameters $\{\beta_k\}_{k \geq -1}$, and test points $\{x_k\}_{k \geq 0}$.
Then, the relation $(\rel_0)\equiv (\nrel_0)$ is satisfied with $\hat{x}_0:=x_0$ and
\begin{equation}\label{nonsm-rel-0}
C_0 := \frac{1}{2}\frac{\lambda_0^2}{\sigma_d(\lambda_0\sigma_f + \beta_{-1})}\norm{g_0}_*^2.
\end{equation}
%
\noindent
(ii)
Consider a structured problem in the class $\mathcal{SP}(m_f,\sigma_f,\bar\sigma_f,L,\delta)$ and let $\{(\aphi_k(x),\apsi_k(x))\}_{k\geq -1}$ be a coupled sequence of auxiliary functions satisfying Property \ref{framework-double} associated with weight parameters $\{\lambda_k\}_{k \geq 0}$, scaling parameters $\{\beta_k\}_{k \geq -1}$, and test points $\{x_k\}_{k \geq 0}$.
Then, the relation $(\rel_0)\equiv(\srel_0)$ is satisfied with $\hat{x}_0:=\w_0$ and
\begin{equation}\label{struc-rel-0}
C_0 := \lambda_0\left(\frac{L(x_0)}{2}-\frac{\sigma_d}{2}\left(\bar\sigma_f+\frac{\beta_{-1}}{\lambda_0}\right)\right)\norm{\w_0-x_0}^2 + \lambda_0 \delta(x_0,\hat{x}_0).
\end{equation}
\end{lem}

\begin{proof}
Note that (B0) implies that $\aphi_k(\sz_k)=\min_{x \in Q}\aphi_k(x)\geq \min_{x \in Q}\apsi_k(x)=\apsi_k(\w_k)$.
Since $\{\beta_k\}$ is non-decreasing, using (B2) with $x=\w_{k+1}$ yields that
\begin{eqnarray}
\apsi_{k+1}(\w_{k+1}) & \geq & \aphi_k(\z_k)+\lambda_{k+1}m_f(x_{k+1};\w_{k+1})+(\beta_{k}+S_k\sigma_f)\xi(\z_k,\w_{k+1})\nonumber\\
& \geq & \apsi_k(\w_k)+\lambda_{k+1}m_f(x_{k+1};\w_{k+1})+(\beta_{k}+S_k\sigma_f)\xi(\z_k,\w_{k+1}) \label{min-aux-rel}
\end{eqnarray}
for every $k \geq -1$. In the case $k=-1$, the conditions (B1), $S_{-1}=0$, and $\z_{-1}=x_0$ lead (\ref{min-aux-rel}) to
\begin{eqnarray}
\apsi_0(\w_0) & \geq & \lambda_0[m_f(x_0;\w_0) - \sigma\xi(x_0,\w_0) + (\sigma+\beta_{-1}/\lambda_0)\xi(x_0,\w_0)]\nonumber\\
&\geq& \lambda_0\left[ m_f(x_0;\w_0) - \sigma\xi(x_0,\w_0) + \frac{\sigma_d}{2}\left(\sigma + \frac{\beta_{-1}}{\lambda_0}\right)\norm{\w_0-x_0}^2 \right]
\label{rel-0-proof-gen}
\end{eqnarray}
for any $\sigma\geq 0$.
Let us firstly show (ii). Letting $\sigma:=\bar\sigma_f$, the settings $\hat{x}_0=\w_0$ and (\ref{struc-rel-0}) yields
\[ \apsi_0(\w_0) + C_0 \stackrel{(\ref{rel-0-proof-gen})}{\geq} \lambda_0\left[ m_f(x_0;\w_0) - \bar\sigma_f\xi(x_0,\hat{x}_0) + \frac{L(x_0)}{2}\norm{\hat{x}_0-x_0}^2 + \delta(x_0,\hat{x}_0) \right] \geq \lambda_0 f(\hat{x}_0) \]
which proves the relation $(\rel_0)$.

It remains to prove (i). By the definition of $m_f(\cdot;\cdot)$ for the non-smooth case, the inequality (\ref{rel-0-proof-gen}) with $\sigma:=\sigma_f$ implies
\begin{eqnarray*}
\apsi_0(\w_0) &\stackrel{(\ref{rel-0-proof-gen})}{\geq}& \lambda_0 \left[ f(x_0) + \innprod{g_0}{\w_0-x_0} + \frac{\sigma_d}{2}\left(\sigma_f + \frac{\beta_{-1}}{\lambda_0}\right)\norm{\w_0-x_0}^2 \right]\\
&=&\lambda_0 f(x_0) + \innprod{\lambda_0 g_0}{\w_0-x_0} + \frac{\sigma_d}{2}\left(\lambda_0\sigma_f + \beta_{-1}\right)\norm{\w_0-x_0}^2 \\
&\geq& \lambda_0f(x_0) - \frac{1}{2}\frac{\lambda_0^2}{\sigma_d (\lambda_0\sigma_f + \beta_{-1})}\norm{g_0}_*^2,
\end{eqnarray*}
where the last inequality is due to the basic fact
\begin{equation}\label{Young's-ineq}
\frac{1}{2}\norm{x}^2 + \frac{1}{2}\norm{s}_*^2 \geq \innprod{s}{x} \textrm{ for } x \in E, ~s \in E^*.
\end{equation}
This means that the relation $(\rel_0)$ is satisfied with the setting $\hat{x}_0 = x_0$ and (\ref{nonsm-rel-0}).
\end{proof}


\subsection{Validity of $(\rel_k),~(\nrel_k)$, and $(\srel_k)$ for the classical method when $k > 0$}

Let us complete our induction for the classical method. The items (i) and (ii) in the following lemma correspond to the $k$-th iteration of the classical method in Methods \ref{gen-alg-nonsm} and \ref{gen-alg-struc}, respectively.

\begin{lem}\label{rel-k-clas}
%
(i)
Consider a non-smooth problem in the class $\mathcal{NSP}(g, \sigma_f)$ and let $\{(\aphi_k(x),\apsi_k(x))\}_{k\geq -1}$ be a coupled sequence of auxiliary functions satisfying Property \ref{framework-double} associated with weight parameters $\{\lambda_k\}_{k \geq 0}$, scaling parameters $\{\beta_k\}_{k \geq -1}$, and test points $\{x_k\}_{k \geq 0}$.
Suppose for $k \geq 0$ that the relation $(\rel_k)$ is satisfied for some $\hat{x}_k\in Q$, $C_k \geq 0$.
If the relation $x_{k+1}=\z_k$ holds, then the relation $(\rel_{k+1})$ is satisfied with $\hat{x}_{k+1}:=\frac{S_k\hat{x}_k+\lambda_{k+1}x_{k+1}}{S_{k+1}}$ and
\begin{equation}\label{nonsm-rel-k-clas}
C_{k+1}:=C_k + \frac{1}{2\sigma_d}\frac{\lambda_{k+1}^2}{\beta_k + S_{k+1}\sigma_f}\norm{g_{k+1}}_*^2.
\end{equation}
Furthermore, if $(\nrel_k)$ is satisfied, then so is $(\nrel_{k+1})$ with the same settings of $x_{k+1}$ and $C_{k+1}$.
%
\\[3truemm]
(ii)
Consider a structured problem in the class $\mathcal{SP}(m_f,\sigma_f,\bar\sigma_f,L,\delta)$ and let $\{(\aphi_k(x),\apsi_k(x))\}_{k\geq -1}$ be a coupled sequence of auxiliary functions satisfying Property \ref{framework-double} associated with weight parameters $\{\lambda_k\}_{k \geq 0}$, scaling parameters $\{\beta_k\}_{k \geq -1}$, and test points $\{x_k\}_{k \geq 0}$.
Suppose for $k \geq 0$ that the relation $(\rel_k)$ is satisfied for some $\hat{x}_k\in Q$, $C_k \geq 0$.
If the relation $x_{k+1}=\z_k$ holds, then the relation $(\rel_{k+1})$  is satisfied with $\hat{x}_{k+1}:=\frac{S_k\hat{x}_k+\lambda_{k+1}\w_{k+1}}{S_{k+1}}$ and
\begin{equation*}
C_{k+1} := C_k + \lambda_{k+1}\left( \frac{L(x_{k+1})}{2} - \frac{\sigma_d}{2}\left(\bar\sigma_f+\frac{\beta_k + S_k\sigma_f}{\lambda_{k+1}}\right) \right)\norm{\w_{k+1}-x_{k+1}}^2 +\lambda_{k+1}\delta(x_{k+1},\w_{k+1}).
\end{equation*}
Furthermore, if $(\srel_k)$ is satisfied, then so is $(\srel_{k+1})$ with the same settings of $x_{k+1}$ and $C_{k+1}$.
\end{lem}

\begin{proof}
Using (\ref{min-aux-rel}) and the relation $x_{k+1}=\z_k$ imply for any $\sigma \geq 0$ that
\begin{eqnarray*}
\apsi_{k+1}(\w_{k+1})
&\geq&\apsi_k(\w_k) + \lambda_{k+1}m_f(x_{k+1};\w_{k+1}) + (\beta_k + S_k\sigma_f)\xi(\z_k,\w_{k+1})\\
&=& \apsi_k(\w_k) \\&&~~+ \lambda_{k+1}\left([m_f(x_{k+1};\w_{k+1})-\sigma\xi(x_{k+1},\w_{k+1})] + \left(\sigma + \frac{\beta_k + S_k\sigma_f}{\lambda_{k+1}}\right)\xi(x_{k+1},\w_{k+1}) \right)\\
&\geq& \apsi_k(\w_k) \\&& + \lambda_{k+1}\left( [m_f(x_{k+1};\w_{k+1})-\sigma\xi(x_{k+1},\w_{k+1})] + \frac{\sigma_d}{2}\left(\sigma+\frac{\beta_k + S_k\sigma_f}{\lambda_{k+1}}\right)\norm{\w_{k+1}-x_{k+1}}^2 \right).
\end{eqnarray*}
For the structured problems, letting $\sigma:=\bar\sigma_f$ and the definition of $C_{k+1}$ in (ii) yield that
\[ \apsi_{k+1}(\w_{k+1})+C_{k+1} \geq \apsi_k(\w_k) + C_k + \lambda_{k+1}f(\w_{k+1}). \]
Using $(\rel_k)$ and the convexity of $f$ conclude the relation $(\rel_{k+1})$; $(\srel_{k+1})$ follows by using $(\srel_k)$ and the inequality above. Hence, the assertion (ii) is proved.

For the non-smooth problems, on the other hand, we can continue by taking $\sigma:=\sigma_f$ as follows.
\begin{eqnarray*}
\apsi_{k+1}(\w_{k+1}) &\geq& \apsi_k(\w_k) + \lambda_{k+1}f(x_{k+1}) + \innprod{\lambda_{k+1}g_{k+1}}{\w_{k+1}-x_{k+1}} + \frac{\sigma_d}{2}(\beta_k + S_{k+1}\sigma_f)\norm{\w_{k+1}- x_{k+1}}^2\\
&\stackrel{(\ref{Young's-ineq})}{\geq}& \apsi_k(\w_k) + \lambda_{k+1}f(x_{k+1}) - \frac{1}{2}\frac{\lambda_{k+1}^2}{\sigma_d(\beta_k + S_{k+1}\sigma_f)}\norm{g_{k+1}}_*^2.
\end{eqnarray*}
Hence, the definition (\ref{nonsm-rel-k-clas}) of $C_{k+1}$ yields that
\[ \apsi_{k+1}(\w_{k+1}) + C_{k+1} \geq \apsi_k(\w_k) + C_k + \lambda_{k+1}f(x_{k+1}). \]
Now the assertion (i) follows by the same way as (ii).
\end{proof}


\subsection{Validity of $(R_k)$ for the modified method when $k > 0$}

The following lemma completes our induction 
for the modified method. In a similar manner as Lemma \ref{rel-k-clas}, the items (i) and (ii) below correspond to the $k$-th iteration of the modified method in Method \ref{gen-alg-nonsm} and \ref{gen-alg-struc}, respectively.

\begin{lem}\label{rel-k-modif}
%
(i)
Consider a non-smooth problem in the class $\mathcal{NSP}(g,\sigma_f)$ and let $\{(\aphi_k(x),\apsi_k(x))\}_{k\geq -1}$ be a coupled sequence of auxiliary functions satisfying Property \ref{framework-double} associated with weight parameters $\{\lambda_k\}_{k \geq 0}$, scaling parameters $\{\beta_k\}_{k \geq -1}$, and test points $\{x_k\}_{k \geq 0}$.
Suppose for $k \geq 0$ that the relation $(\rel_k)$ is satisfied for some $\hat{x}_k\in Q$, $C_k \geq 0$.
If the relation
$x_{k+1} = \frac{S_k\hat{x}_k + \lambda_{k+1}\z_{k}}{S_{k+1}}$
holds, then the relation $(\rel_{k+1})$ is satisfied with $\hat{x}_{k+1}:=x_{k+1}$ and
\begin{equation}\label{nonsm-rel-k-modif}
C_{k+1} := C_k + \frac{1}{2\sigma_d} \frac{\lambda_{k+1}^2 S_{k+1}}{\lambda_{k+1}^2\sigma_f + S_{k+1}(\beta_k + S_k \sigma_f)}\norm{g_{k+1}}_*^2.
\end{equation}
%

\noindent
(ii)
Consider a structured problem in the class $\mathcal{SP}(m_f,\sigma_f,\bar\sigma_f,L,\delta)$ and let $\{(\aphi_k(x),\apsi_k(x))\}_{k\geq -1}$ be a coupled sequence of auxiliary functions satisfying Property \ref{framework-double} associated with weight parameters $\{\lambda_k\}_{k \geq 0}$, scaling parameters $\{\beta_k\}_{k \geq -1}$, and test points $\{x_k\}_{k \geq 0}$.
Suppose for $k \geq 0$ that the relation $(\rel_k)$ is satisfied for some $\hat{x}_k\in Q$, $C_k \geq 0$.
If the relations
$ x_{k+1} = \frac{S_k\hat{x}_k + \lambda_{k+1}\z_{k}}{S_{k+1}}$ and $\hat{x}_{k+1}=\frac{S_k\hat{x}_k + \lambda_{k+1}\w_{k+1}}{S_{k+1}} $
hold, then
the relation $(\rel_{k+1})$ is satisfied with
\begin{equation}\label{struc-rel-k-modif}
C_{k+1}:=C_k + S_{k+1}\left(\frac{L(x_{k+1})}{2}-\frac{\sigma_d}{2} \left(\bar\sigma_f + \frac{S_{k+1}(\beta_k + S_k \sigma_f)}{\lambda_{k+1}^2} \right)\right)\norm{\hat{x}_{k+1}-x_{k+1}}^2 + S_{k+1}\delta(x_{k+1},\hat{x}_{k+1}).\end{equation}
\end{lem}

\begin{proof}
Denote $x'_{k+1} := \frac{S_k\hat{x}_k + \lambda_{k+1}\w_{k+1}}{S_{k+1}}$.
If $x_{k+1} = \frac{S_k\hat{x}_k + \lambda_{k+1}\z_{k}}{S_{k+1}}$ holds, then $x'_{k+1}-x_{k+1}=\frac{\lambda_{k+1}}{S_{k+1}}(\w_{k+1}-\z_k)$.
Using (\ref{min-aux-rel}) and the relation $(\rel_k)$, we have
\begin{eqnarray}
\apsi_{k+1}(\w_{k+1}) + C_k &\geq& \apsi_k(\w_k) + C_k + \lambda_{k+1}m_f(x_{k+1};\w_{k+1}) + (\beta_k + S_k\sigma_f)\xi(\z_k,\w_{k+1})\nonumber\\
&\geq& S_kf(\hat{x}_k) + \lambda_{k+1}m_f(x_{k+1};\w_{k+1}) + (\beta_k + S_k\sigma_f)\xi(\z_k,\w_{k+1}) \nonumber\\
&\geq& S_km_f(x_{k+1};\hat{x}_k) + \lambda_{k+1}m_f(x_{k+1};\w_{k+1}) + (\beta_k + S_k\sigma_f)\xi(\z_k,\w_{k+1}) \nonumber\\
&\geq& S_{k+1}m_f(x_{k+1};x'_{k+1}) + (\beta_k + S_k\sigma_f)\xi(\z_k,\w_{k+1}), \label{rel-k-modif-1}
\end{eqnarray}
where we used $f(x)\geq m_f(y;x),\forall x,y\in Q$ and the convexity of $m_f(x_{k+1};\cdot)$ for the last two inequalities.
Since $\xi(\z_k,\w_{k+1}) \geq \frac{\sigma_d}{2}\norm{\w_{k+1}-\z_k}^2=\frac{\sigma_d}{2}\frac{S_{k+1}^2}{\lambda_{k+1}^2}\norm{x'_{k+1}-x_{k+1}}^2$ and
\begin{eqnarray*}
m_f(x_{k+1};x'_{k+1})
&=&
m_f(x_{k+1};x'_{k+1})-\sigma\xi(x_{k+1},x'_{k+1})+\sigma\xi(x_{k+1},x'_{k+1})\\
&\geq&
m_f(x_{k+1};x'_{k+1})-\sigma\xi(x_{k+1},x'_{k+1})+\frac{\sigma\sigma_d}{2}\norm{x_{k+1}-x'_{k+1}}^2
\end{eqnarray*}
hold for any $\sigma \geq 0$, the inequality (\ref{rel-k-modif-1}) implies that
\begin{eqnarray}
\apsi_{k+1}(\w_{k+1})+C_k &\geq&
S_{k+1}[m_f(x_{k+1};x'_{k+1})-\sigma\xi(x_{k+1},x'_{k+1})]\nonumber\\
&&\qquad + \frac{\sigma_d}{2} S_{k+1}\left(\sigma + \frac{S_{k+1}(\beta_k + S_k \sigma_f)}{\lambda_{k+1}^2} \right)\norm{x'_{k+1}-x_{k+1}}^2. \label{rel-k-modif-2}
\end{eqnarray}
Let us prove (ii) at first. Since $\hat{x}_{k+1}=x'_{k+1}$ by the assumption, adding
\[ S_{k+1}\left(\frac{L(x_{k+1})}{2}-\frac{\sigma_d}{2} \left(\bar\sigma_f + \frac{S_{k+1}(\beta_k + S_k \sigma_f)}{\lambda_{k+1}^2} \right)\right)\norm{\hat{x}_{k+1}-x_{k+1}}^2 + S_{k+1}\delta(x_{k+1},\hat{x}_{k+1})\]
to both sides in (\ref{rel-k-modif-2}) with $\sigma:=\bar\sigma_f$ and using the inequality (\ref{struc-ineq}) implies the relation $(\rel_{k+1})$ with the setting (\ref{struc-rel-k-modif}).

To prove (i), on the other hand, letting $\sigma:=\sigma_f$ and using $m_f(x_{k+1};x'_{k+1})-\sigma\xi(x_{k+1},x'_{k+1})=f(x_{k+1})+\innprod{g_{k+1}}{x'_{k+1}-x_{k+1}}$ leads (\ref{rel-k-modif-2}) to
\begin{eqnarray*}
\apsi_{k+1}(\w_{k+1}) + C_k&\geq& S_{k+1}f(x_{k+1}) + \innprod{S_{k+1}g_{k+1}}{x'_{k+1}-x_{k+1}}\\
&&\qquad + \frac{\sigma_d}{2} S_{k+1}\left(\sigma_f + \frac{S_{k+1}(\beta_k + S_k \sigma_f)}{\lambda_{k+1}^2} \right)\norm{x'_{k+1}-x_{k+1}}^2\\
&\stackrel{(\ref{Young's-ineq})}{\geq}& S_{k+1}f(x_{k+1}) - \frac{1}{2}\frac{S_{k+1}^2}{\sigma_d S_{k+1}\left(\sigma_f + \frac{S_{k+1}(\beta_k + S_k \sigma_f)}{\lambda_{k+1}^2} \right) }\norm{g_{k+1}}_*^2\\
&=& S_{k+1}f(x_{k+1}) - \frac{1}{2\sigma_d} \frac{\lambda_{k+1}^2 S_{k+1}}{\lambda_{k+1}^2\sigma_f + S_{k+1}(\beta_k + S_k \sigma_f)}\norm{g_{k+1}}_*^2.
\end{eqnarray*}
This means that the relation $(\rel_{k+1})$ is obtained with (\ref{nonsm-rel-k-modif}).
\end{proof}


\subsection{Proof of Theorems~\ref{gen-conv-nonsm} and \ref{gen-conv-struc}}

Let us show Theorem \ref{gen-conv-nonsm}; the proof of Theorem \ref{gen-conv-struc} is analogue replacing $(\nrel_k)$ with $(\srel_k)$ and the part (i) with (ii) in Lemmas \ref{rel-0}, \ref{rel-k-clas}, \ref{rel-k-modif}.

By the description of the Method \ref{gen-alg-nonsm}, we can apply part (i) of each Lemmas \ref{rel-0},\ref{rel-k-clas},\ref{rel-k-modif} to show that the relation $(\rel_k)$ holds for every $k \geq 0$ with $C_k$ defined by (\ref{C_k-nonsm}); for the classical method, the relation $(\nrel_k)$ can also be verified. The assertion follows from Lemma \ref{R_k-bound} and its analogue for the relation $(\nrel_k)$ (see Remark \ref{remark-R_k} (1)). \qed

We remark that the above lemmas justify our choices for the update
formulas of $x_k$ and $\hat{x}_k$ in Methods~\ref{gen-alg-nonsm} and~\ref{gen-alg-struc}. In fact, what is behind the proofs is the satisfaction of the relation $(\rel_k)$ (or its variants). Therefore, the relation $(\rel_k)$ is an implicit factor in our unifying framework.


\section{Optimal/nearly optimal convergence rates of (sub)gradient-based methods}
\label{sec-convergence}

In this section, we finally give the actual convergence rates for Methods \ref{gen-alg-nonsm} and \ref{gen-alg-struc} based on the general estimates presented in Section \ref{sec-general-convergence}, and compare these results with the existing ones. Our choices for weight $\{\lambda_k\}_{k \geq 0}$ and scaling parameters $\{\beta_k\}_{k \geq -1}$ resemble and extend the existing ones to compute approximate solutions $\{\hat{x}_k\}_{k \geq 0}$.

As a matter of comparison, we summarize the optimal convergence rates for each problem classes given in Sections \ref{ssec-existing-nonsm} and \ref{sssec-examples-struc} at Table \ref{opt-conv-rate}. 
This table shows the optimal convergence rates of $f(\hat{x}_k)-f(x^*)$ for PGMs applied to non-smooth, smooth, and weakly smooth problems (remark that $\sigma_d\sigma_f$ becomes a convexity parameter of $f$ with respect to the norm $\norm{\cdot}$; see Section \ref{ssec-problem-struc}).


\begin{table}[htbp]
\caption{Optimal convergence rates of PGMs. Here $\sigma_f \in \sigma(f)$, $k$ is an iteration counter, and $c_1(\cdot)$ and $c_2(\cdot)$ are fixed continuous functions. Refer to examples (i) and (iv) in Section \ref{sssec-examples-struc} for the descriptions of smooth and weakly smooth problems, respectively.}
\scalebox{0.9}{
\begin{tabular}{|c|c|c|}
\hline
problem class / type of convexity
& non-strongly convex ($\sigma_f=0$) & strongly convex ($\sigma_f>0$)\\
\hline
	
	non-smooth problem with (\ref{subgrad-bounded}) for some $M>0$ &
	$O\left(M\sqrt{\frac{d(x^*)}{\sigma_d k}}\right)$ &
	$O\left(\frac{M^2}{\sigma_d\sigma_f k}\right)$\\
	
	smooth problem $C_L^{1,1}(Q)$ &
	$O\left(\frac{Ld(x^*)}{\sigma_d k^2}\right)$ &
	$O\left(\exp\left(-\sqrt{\frac{L}{\sigma_d\sigma_f}}k\right)\right)$ \\
	
	weakly smooth problem $C_M^{1,\rho-1}(Q)$, $\rho\in[1,2)$ &
	$c_1(\rho)M\left(\frac{d(x^*)}{\sigma_d}\right)^{\frac{\rho}{2}}k^{-\frac{3\rho-2}{2}}$ &
	$c_2(\rho)\left(\frac{M^2}{(\sigma_d\sigma_f)^\rho}k^{-(3\rho-2)}\right)^{\frac{1}{2-\rho}}$\\
\hline
\end{tabular}
} 
\label{opt-conv-rate}
\end{table}

For CGMs applied to weakly smooth problems (the class $C_M^{1,\rho-1}(Q)$, $\rho \in (1,2]$), the convergence rate
\begin{equation}\label{CGM-known-rate}
f(\hat{x}_k)-f(x^*) \leq O\left(\frac{M{\rm Diam(Q)}^{\rho}}{k^{\rho-1}}\right)
\end{equation}
can be achievable using the classical one (\ref{classical-CGM}) or some of its variants \cite{Nes15}.
This rate is known to be optimal when $\rho=2$ in the sense of linear optimization oracle \cite{Lan14a} and nearly optimal otherwise \cite{GN15}.

We show optimal convergence results of PGMs for the non-smooth problems in the next subsection, for the structured problems with inexact oracle in Sections \ref{ssec-eff-clas-struc}, \ref{ssec-eff-modif-struc}, and for the weakly smooth problems in the last subsection, all for the strongly convex cases. Optimal and nearly optimal convergences of CGMs are developed in Sections \ref{ssec-eff-modif-struc} and \ref{sssec-CGM-weak}.

All of convergence rates matches the known optimal rates of convergence (excepting the classical method for the structured problems).

A noteworthy new result is the attainment of the optimal convergence rate for weakly smooth problems in the strongly convex case with less prior information of the objective function than the existing ones (Section \ref{sssec-convergence-st}). In addition, for smooth problems, the obtained convergence rates slightly improve the existing ones (Sections \ref{ssec-eff-clas-struc} and \ref{ssec-eff-modif-struc}).

Another consequence is that the existing methods included in our unifying framework can be naturally extended for wider classes of problems. In particular, without using a multistage procedure, the DAM for the non-smooth problems can be extended to the strongly convex case (Section \ref{ssec-eff-nonsm}), and Nesterov's and Tseng's PGMs can be extended to the weakly smooth and/or the strongly convex cases (Sections \ref{ssec-eff-modif-struc}, \ref{ssec-eff-modif-weak}).


\subsection{Optimal convergence rate for non-smooth problems}
\label{ssec-eff-nonsm}

Let us analyze the convergence rate of PGMs yielded from Method \ref{gen-alg-nonsm}.
Recall that Method \ref{gen-alg-nonsm} generates a sequence $\{\hat{x}_k\}$ which satisfies the relation $(\rel_k)$ with $C_k$ defined by (\ref{C_k-nonsm}).

When $\sigma_f = 0$, the definitions of $C_k$ for the classical and the modified methods become the same: $C_k=\frac{1}{2\sigma_d}\sum_{i=0}^k\frac{\lambda_i^2}{\beta_{i-1}}\norm{g_i}_*^2$; this case is analyzed in \cite[Corollary 11]{IF14} which ensures the optimal convergence rate $O(M\sqrt{d(x^*)/(\sigma_dk)}$ with an advantage that we do not need values $d(x^*)$ and $M$ in the definition of the parameters $\{\lambda_k\}$ and $\{\beta_k\}$ to achieve $O(1/\sqrt{k})$-convergence.

When $\sigma_f >0$, note that
\[ \frac{\lambda_i^2S_{i}}{\lambda_{i}^2\sigma_f + S_i(\beta_{i-1} + S_{i-1}\sigma_f)}  = \frac{\lambda_i^2}{\beta_{i-1}+S_{i-1}\sigma_f+\frac{\lambda_i^2}{S_i}\sigma_f} \geq \frac{\lambda_i^2}{\beta_{i-1} + S_i\sigma_f} \]
holds since $\lambda_i/S_i\leq 1$.
In this case, theoretically, the classical method ensures not a worse convergence rate than the modified counterpart.

We give an optimal convergence result with a simple choice for the parameters $\lambda_k=(k+1)/2$ and $\beta_k\equiv 0$ below. Note that every subproblem $\min_{x \in Q}\sphi_k(x)$ has a unique solution even if $\beta_k \equiv 0$ because $\sigma(\sphi_k) \ni \beta_k + S_k\sigma_f = S_k \sigma_f > 0$ (see the proof of Lemma \ref{prop-subprob}).

\begin{thm} 
Consider a non-smooth problem in the class $\mathcal{NSP}(g,\sigma_f)$.
Let $\{(\sz_{k-1},x_k,g_k,\hat{x}_k)\}_{k \geq 0}$ be generated by Method \ref{gen-alg-nonsm} associated with $\lambda_k = (k+1)/2$ and $\beta_k \equiv 0$.
Assume that $\sigma_f>0$ and $\sup_{k \geq 0}\norm{g_k}_* \leq M_f < +\infty$. Then, we have
\[ \max\{f(\hat{x}_k)-f(x^*),~\min_{0\leq i\leq k}f(x_i)-f(x^*)\}+\sigma_f\xi(x_{k+1},x^*) \leq \frac{2M_f^2}{\sigma_d\sigma_f (k+4)}, \quad \forall k \geq 0 \]
with the classical method, and
\[ f(\hat{x}_k)-f(x^*)+\sigma_f\xi(\sz_k,x^*) \leq \frac{2M_f^2}{\sigma_d\sigma_f}\frac{k+\log k +3/2}{(k+1)(k+2)} = O\left(\frac{M_f^2}{\sigma_d\sigma_f k}\right),\quad \forall k \geq 1 \]
with the modified method.
\end{thm}

\begin{proof}
Since $\beta_k\equiv 0$ and $S_k=\frac{(k+1)(k+2)}{4}$, Theorem \ref{gen-conv-nonsm} implies the estimate
\begin{equation}\label{est-nonsm-st-1}
f(\hat{x}_k)-f(x^*)+\sigma_f\xi(\sz_k,x^*) \leq \frac{C_k}{S_k} = \frac{4C_k}{(k+1)(k+2)}
\end{equation}
with $C_k$ defined by (\ref{C_k-nonsm}).
The classical method also admits the same estimate replacing $f(\hat{x}_k)-f(x^*)$ by $\min_{0\leq i \leq k}f(x_i)-f(x^*)$ and we have
\[ C_k = \frac{1}{2\sigma_d}\sum_{i=0}^k\frac{\lambda_i^2}{\beta_{i-1}+S_i\sigma_f}\norm{g_i}_*^2 \leq \frac{M_f^2}{2\sigma_d\sigma_f}\sum_{i=0}^k\frac{\lambda_i^2}{S_i}.\]
Using the inequality
\begin{equation}\label{FG-prop}
\sum_{i=0}^k\frac{\lambda_i^2}{S_i} = \sum_{i=0}^k\frac{i+1}{i+2} \leq \frac{(k+1)(k+2)}{k+4}
\end{equation}
(see \cite[Proposition 7.3]{FG14}), we obtain the first assertion for the classical method.

In the modified method, on the other hand, we have
\[ C_k = \frac{1}{2\sigma_d}\sum_{i=0}^k \frac{\lambda_i^2S_{i}}{\lambda_{i}^2\sigma_f + S_i(\beta_{i-1} + S_{i-1}\sigma_f)}\norm{g_i}_*^2 \leq \frac{M_f^2}{2\sigma_d\sigma_f}\sum_{i=0}^k\frac{(i+1)(i+2)}{i(i+2)+4} \]
and
\[ \sum_{i=0}^k\frac{(i+1)(i+2)}{i(i+2)+4}\leq \frac{1}{2}+\sum_{i=1}^k\frac{(i+1)(i+2)}{i(i+2)} = \frac{1}{2}+\sum_{i=1}^k\left(1+\frac{1}{i}\right) \leq \frac{1}{2}+k+(1+\log k) \]
for all $k \geq 1$, which leads (\ref{est-nonsm-st-1}) to the second assertion.
\end{proof}

Note that the choices of parameters $\lambda_k=(k+1)/2$ and $\beta_k\equiv 0$ do not depend on $M_f$ and $\sigma_f$. However, we need $\sigma_f$ when we solve the subproblems. For instance, the classical method with the extended MD model (\ref{eMD-model}) associated with the above parameters becomes
\begin{eqnarray*}
x_{k+1}:=\sz_k &:=& \argmin_{x \in Q}\{\lambda_k[f(x_k)+\innprod{g_k}{x-x_k}+\sigma_f\xi(x_k,x)]+S_{k-1}\sigma_f\xi(x_k,x)\}\\
&=& \argmin_{x \in Q}\{\lambda_k[f(x_k)+\innprod{g_k}{x-x_k}]+S_k\sigma_f\xi(x_k,x)\}\\
&=& \argmin_{x \in Q}\left\{\frac{\lambda_k}{S_k\sigma_f}[f(x_k)+\innprod{g_k}{x-x_k}]+\xi(x_k,x)\right\}\\
&=& \argmin_{x \in Q}\left\{\frac{2}{\sigma_f(k+2)}[f(x_k)+\innprod{g_k}{x-x_k}]+\xi(x_k,x)\right\},\\
\hat{x}_k &:=& \frac{1}{S_k}\sum_{i=0}^k\lambda_ix_i = \frac{2}{(k+1)(k+2)}\sum_{i=0}^k(i+1)x_i,
\end{eqnarray*}
which gives the estimates
\begin{equation}\label{est-nonsm-st-clas}
\begin{array}{rll}
\max\{f(\hat{x}_k)-f(x^*),~\min_{0\leq i\leq k}f(x_i)-f(x^*)\}+\sigma_f\xi(x_{k+1},x^*) &\leq& \ds\frac{2M_f^2}{\sigma_d\sigma_f(k+4)},\\[3truemm]
\min\{\norm{\hat{x}_k-x^*}^2,~\norm{x_{i(k)}-x^*}^2,~\norm{x_{k+1}-x^*}^2\} &\leq& \ds\frac{2M_f^2}{\sigma_d^2\sigma_f^2(k+4)},
\end{array}
\end{equation}
for all $k\geq 0$, where $i(k)\in\Argmin_{0\leq i\leq k}f(x_i)$ (see Lemma \ref{R_k-bound} and Remark \ref{remark-R_k}). Notice that the computation of $\sz_k$ is equivalent to the subproblem (\ref{eMD-subprob-nonst}) (the extended MD model for non-strongly convex case) with $\lambda_k:=\frac{2}{\sigma_f(k+2)}$ and $\beta_k \equiv 1$. This result is closely related to \cite[Theorem 1]{NL14}, \cite[Proposition 3.1]{Bach15}, and \cite[Proposition 2.8]{NB}.

The convergence result (\ref{est-nonsm-st-clas}) is also valid for the DA model (\ref{DA-model}), and then we conclude that a strongly convex version of the DAM achieves the optimal complexity for non-smooth problems (see Section \ref{ssec-existing-nonsm}). This result is new. Note that we do not exploit the multistage procedure and do not require an upper bound of $d(x^*)$ to obtain the optimality as required in \cite{JN14}.


\subsection{Convergence rate of the classical method for structured problems with constants $L$ and $\delta$}
\label{ssec-eff-clas-struc}

We next analyze the convergence rate of PGMs produced by Method \ref{gen-alg-struc} for a particular case of structured problems.
Let us consider a structured problems in $\mathcal{SP}(m_f,\sigma_f,\bar\sigma_f,L,\delta)$ for the particular case $L(\cdot)=L\geq 0$ and $\delta(\cdot,\cdot)=\delta\geq 0$.
In this case, we assume that $L\geq \bar\sigma_f\sigma_d$; notice that, in view of $m_f(y;x)\leq f(x)$ and $\xi(y,x)\geq \frac{\sigma_d}{2}\norm{x-y}^2$ for $x,y \in Q$, the inequality (\ref{struc-ineq}) yields $0\leq (L-\bar\sigma_f\sigma_d)\frac{1}{2}\norm{y-x}^2+\delta$.
We firstly show a convergence result of the classical method of Method \ref{gen-alg-struc} which does not ensure the optimal convergence rate for the class $C_L^{1,1}(Q)$. This rate is as better as the existing PGMs compared in this subsection.

\begin{thm}\label{est-struc-clas-st}
Consider a structured problem in the class $\mathcal{SP}(m_f,\sigma_f,\bar\sigma_f,L,\delta)$. Assume additionally that $L(\cdot)=L\geq 0$, $\delta(\cdot,\cdot)=\delta\geq 0$, and  $L\geq \bar\sigma_f\sigma_d$.
Let $\{(\z_{k-1},\w_{k-1},x_k,\hat{x}_k)\}_{k \geq 0}$ be generated by the classical method of Method \ref{gen-alg-struc} with
\begin{equation}\label{param-struc-clas}
\beta_k \equiv \frac{L-\bar\sigma_f\sigma_d}{\sigma_d},~~\lambda_0 = 1,~~\lambda_{k+1} = \frac{\beta_{k}+S_{k}\sigma_f}{\beta_{k}}.
\end{equation}
Then, for every $k \geq 0$, we have
\begin{equation}\label{bound-struc-clas}
f(\hat{x}_k) - f(x^*) + \sigma_f\xi(\z_k,x^*) \leq \frac{L-\bar\sigma_f\sigma_d}{\sigma_d}l_d(\z_k;x^*)\min\left\{ \left(1-\frac{\sigma_f\sigma_d}{L-\bar\sigma_f\sigma_d+\sigma_f\sigma_d}\right)^k,\frac{1}{k+1} \right\} + \delta.
\end{equation}
Furthermore, the left hand side of (\ref{bound-struc-clas}) can be replaced by $\frac{1}{S_k}\sum_{i=0}^k\lambda_kf(\w_{k})-f(x^*) + \sigma_f\xi(\z_k,x^*)$ or by $\min_{0\leq i\leq k}f(\w_{i})-f(x^*) + \sigma_f\xi(\z_k,x^*)$.
\end{thm}

\begin{proof}
The classical method admits the relation $(\rel_k)$ and $(\srel_k)$ with
\[ C_k = \frac{1}{2}\sum_{i=0}^k \lambda_i\left(L - \sigma_d \left(\bar\sigma_f+ \frac{\beta_{i-1}+S_{i-1}\sigma_f}{\lambda_i} \right)\right)\norm{\w_i-x_i}^2 + \sum_{i=0}^k\lambda_i\delta.\]
The definitions of $\lambda_k$  and $\beta_k$ implies that $C_k=\sum_{i=0}^k\lambda_i\delta=S_k\delta$ (since $\frac{\beta_{i-1}+S_{i-1}\sigma_f}{\lambda_i}=\beta_{i-1}=\frac{L-\bar{\sigma}_f\sigma_d}{\sigma_d}$) and $S_k = 1+ \left(1+\frac{\sigma_f}{\beta_{-1}}\right)S_{k-1}$ for all $k \geq 0$.
Therefore, we have $S_k \geq k+1$ and $S_k \geq (1+\frac{\sigma_f}{\beta_{-1}})^kS_0=(1-\frac{\sigma_f}{\beta_{-1}+\sigma_f})^{-k}$, and the result follows from Theorem~\ref{gen-conv-struc}.
\end{proof}

Notice that the right hand side of (\ref{bound-struc-clas}) goes to $\delta$ as $k \to \infty$.

It is interesting to notice that the particular choice of parameters (\ref{param-struc-clas}) does not necessarily require the knowledge of $\sigma_f$ and $\bar{\sigma}_f$ for the implementation of the classical gradient method with the extended MD model (\ref{eMD-model}); for smooth problems (\ie, $f \in C_L^{1,1}(Q)$), for instance, the corresponding subproblem can be rewritten as follows: 
\begin{eqnarray}
\z_k&:=& \argmin_{x \in Q}\left\{\lambda_k\left[f(x_k)+\innprod{\nabla{f}(x_k)}{x-x_k}+\bar\sigma_f\xi(x_k,x)\right]+\beta_k\xi(x_k,x)+S_{k-1}\sigma_f\xi(x_k,x)\right\} \nonumber\\
&=& \argmin_{x \in Q}\left\{f(x_k)+\innprod{\nabla{f}(x_k)}{x-x_k}+\left(\bar{\sigma}_f+\frac{\beta_{k}+S_{k-1}\sigma_f}{\lambda_k}\right)\xi(x_k,x)\right\} \nonumber\\
&\stackrel{(\ref{param-struc-clas})}{=}& \argmin_{x \in Q}\left\{f(x_k)+\innprod{\nabla{f}(x_k)}{x-x_k}+\frac{L}{\sigma_d}\xi(x_k,x)\right\}, \label{subprob-pgm}
\end{eqnarray}
which requires only $L$; in the Euclidean setting ({\it i.e.}, $\frac{1}{\sigma_d}\xi(x_k,x)=\frac{1}{2}\|x_k-x\|_2^2$), furthermore, the Lipschitz condition (\ref{struc-ineq}) ensures that $f(x_{k+1})\leq f(x_k)$ because $x_{k+1}=z_k$ is given by (\ref{subprob-pgm}).
The classical gradient method with the DA model (\ref{DA-model}) and the hybrid model (\ref{hybrid-model}), on the other hand, do not possess this advantage.

Let us see the corresponding PGMs for other particular structures. 
\begin{itemize}
\item Consider the composite problem $\min_{x \in Q}[f(x)\equiv f_0(x)+\varPsi(x)]$ as the example (ii) in Section \ref{sssec-examples-struc} with the structure $\bar\sigma_f=\sigma_{f_0}=0$ (and thus $\sigma_f=\sigma_{\varPsi}$) in the Euclidean setting (then, $\sigma_d=1$). Choosing parameters by (\ref{param-struc-clas}), the classical gradient methods with the extended MD model and the hybrid model yield the Gradient Method ${\cal GM}(x_0,L)$ and the Dual Gradient Method ${\cal DG}(x_0,L)$ in \cite{Nes13}, respectively (in this case, we do not exploit the procedure to estimate the Lipschitz constant $L$).
Then, Theorem \ref{est-struc-clas-st} improves the convergence rates shown in \cite{Nes13} as follows: The linear convergence factor $1-\frac{\sigma_f}{L+\sigma_f}=\frac{L}{L+\sigma_f}$ provided by (\ref{bound-struc-clas}) is less than the one in \cite[Theorem 5]{Nes13} (because $\frac{L}{L+\sigma_f}\leq\min\{\frac{\gamma L}{\sigma_f},1-\frac{\sigma_f}{4\gamma L}\}$ for any $\gamma > 1$) and the same linear convergence is also valid for the method ${\cal DG}(x_0,L)$ which is not presented in the paper (the linear convergence for the dual gradient method was firstly demonstrated in \cite{DGNs}).
\item For the convex problems with inexact oracle model as the example (iii) in Section \ref{sssec-examples-struc} in the Euclidean setting (then, $\sigma_f = \bar\sigma_f$, $\sigma_d = 1$), the classical gradient method with the extended MD model and the hybrid model yield the primal and the dual gradient methods in \cite{DGNs}, respectively (but the definition (\ref{param-struc-clas}) of $\{\lambda_k\}$ is slightly different from (4.1) and (4.2) in \cite{DGNs}). Because of $\sigma_d = 1$ and  $(L-\bar\sigma_f)l_d(\z_k;x^*) \leq Ld(x^*)=\frac{L}{2}\norm{x_0-x^*}_2^2$, the estimate (\ref{bound-struc-clas}) slightly improves Theorems 4 and 5 in \cite{DGNs} (Since $\sigma_f=\bar\sigma_f$, the factor of linear convergence is the same).
\end{itemize}
Note that the classical gradient method of Method \ref{gen-alg-struc} with the DA model (\ref{DA-model}) can reduce the subproblems of the dual gradient method from two \cite{DGNs,Nes13} to one, preserving the same convergence rate.


\subsection{Optimal convergence rate of the modified method for structured problems with constants $L$ and $\delta$}
\label{ssec-eff-modif-struc}

The modified method of Method \ref{gen-alg-struc} for the structured problem in the particular case $L(\cdot)=L\geq 0,~\delta(\cdot,\cdot)=\delta\geq 0$ can be analyzed as follows. Differently from the classical method, it achieves the optimal convergence rate for the class $C_L^{1,1}(Q)$.
The result below further implies efficient rates for the CGMs, too.

\begin{thm}\label{est-struc-modif-st}
Consider a structured problem in the class $\mathcal{SP}(m_f,\sigma_f,\bar\sigma_f,L,\delta)$. Assume additionally that $L(\cdot)=L\geq 0$, $\delta(\cdot,\cdot)=\delta\geq 0$, and  $L\geq \bar\sigma_f\sigma_d$.
\\
(i) Let $\{(\z_{k-1},\w_{k-1},x_k,\hat{x}_k)\}_{k \geq 0}$ be generated by the modified method of Method \ref{gen-alg-struc} with
\begin{equation}\label{param-struc-modif}
\beta_k \equiv \frac{L-\bar\sigma_f\sigma_d}{\sigma_d},~~\lambda_0=1,~~(L-\bar\sigma_f\sigma_d)\lambda_{k+1}^2 = \sigma_d(S_k\sigma_f+\beta_{k-1})(\lambda_{k+1}+S_k)~(k \geq 0)
\end{equation}
(\ie, $\lambda_{k+1}$ is determined as the largest root of the above quadratic equation).
Then, for every $k \geq 0$, we have
\begin{eqnarray*}
f(\hat{x}_k)-f(x^*) + \sigma_f\xi(\z_k,x^*) &\leq& \frac{L-\bar\sigma_f\sigma_d}{\sigma_d}l_d(\z_k;x^*)\min\left\{ \frac{4}{(k+2)^2}, \left(1+\frac{1}{2}\sqrt{\frac{\sigma_f\sigma_d}{L-\bar\sigma_f\sigma_d}}\right)^{-2k} \right\}\\
&&\quad + \min\left\{\frac{1}{3}k +\frac{1}{6}\log(k+2)+1, ~1+\sqrt{\frac{L-\bar\sigma_f\sigma_d}{\sigma_f\sigma_d}}\right\}\delta.
\end{eqnarray*}
(ii) Suppose further that $\sigma_f = 0$ and $Q$ is bounded. Let $\{(\z_{k-1},\w_{k-1},x_k,\hat{x}_k)\}_{k \geq 0}$ be generated by the modified method of Method \ref{gen-alg-struc} with $\beta_k\equiv0$, $\lambda_k:=(k+1)/2$ as a CGM (refer Remark \ref{CGM-gen-bound}). Then, for every $k\geq 0$, we have
\[ f(\hat{x}_k)-f(x^*) \leq \frac{2L\max_{0 \leq i \leq k}\norm{\w_i-\z_{i-1}}^2}{k+4} + \frac{k+3}{3}\delta. \]
\end{thm}

\begin{proof}
By Theorem \ref{gen-conv-struc}, we have the estimate (\ref{gen-est-struc}) with
\begin{eqnarray*}
C_k &=& \frac{1}{2}\sum_{i=0}^k S_i\left(L(x_i)-\sigma_d\left(\bar\sigma_f + \frac{S_{i}(\beta_{i-1}+S_{i-1}\sigma_f)}{\lambda_i^2}\right)\right)\norm{\hat{x}_i-x_i}^2+\sum_{i=0}^k S_i\delta(x_i,\hat{x}_i)\\
	&=& \frac{1}{2}\sum_{i=0}^k \frac{\lambda_i^2}{S_{i}}\left(L-\sigma_d\left(\bar\sigma_f + \frac{S_{i}(\beta_{i-1}+S_{i-1}\sigma_f)}{\lambda_i^2}\right)\right)\norm{\w_{i}-\z_{i-1}}^2 + \sum_{i=0}^kS_i\delta.
\end{eqnarray*}
(i) Notice that, since $\lambda_{k+1}+S_k = S_{k+1}$, (\ref{param-struc-modif}) eliminates the above first summation so that we have $C_k=\sum_{i=0}^kS_i\delta$. Therefore, using Lemmas \ref{lem-S_k-lower-bound-b} to \ref{lem-S_k-ub-nonst}, given at Appendix, for the analysis of (\ref{param-struc-modif}), (\ref{gen-est-struc}) leads to the assertion.

\noindent(ii) Letting $\lambda_k = (k+1)/2$, $\beta_k =0$, and $\sigma_f = 0$ in Theorem \ref{gen-conv-struc} with $C_k$ described above and using the inequality (\ref{FG-prop}) establish that
\[
f(\hat{x}_k)-f(x^*)
\leq
\frac{C_k}{S_k} = \frac{L\sum_{i=0}^k\frac{\lambda_i^2}{S_i}\norm{\w_i-\z_{i-1}}^2}{2S_k} + \frac{\sum_{i=0}^kS_i\delta}{S_k}
\leq
\frac{2L\max_{0 \leq i \leq k}\norm{\w_i-\z_{i-1}}^2}{k+4} + \frac{k+3}{3}\delta.
\]
\end{proof}

When $\delta>0$, the bounds obtained in Theorem  \ref{est-struc-modif-st} (i) and (ii) diverge as $k \to \infty$ unless $\sigma_f>0$ (strongly convex case) for the assertion (i).
Thus, the parameter $\delta\geq 0$ must be sufficiently small in order to ensure an approximate solution with a  desired precision. One can see further discussions on these bounds in \cite{DGNs,DGN}.

In the non-strongly convex case $\sigma_f=\bar\sigma_f =0$, Tseng's PGMs \cite{Tseng10} are derived from the modified method with the model (\ref{eMD-model}) or (\ref{DA-model}) and Nesterov's PGM \cite{Nes05s} is derived with the hybrid model (\ref{hybrid-model}). From these facts, one can conclude that the first result of Theorem \ref{est-struc-modif-st} yields the strongly convex versions of Tseng's and Nesterov's PGMs with optimal complexity (see \cite{DGNs} for the verification of the optimality).
The fast/accelerated gradient method in \cite{DGNs,DGN,Nes13} for strongly convex problems are different from these three particularizations of the models (\ref{eMD-model}) to (\ref{hybrid-model}).

Let us consider the Euclidean setting $d(x)=\frac{1}{2}\norm{x-x_0}_2^2,~\sigma_d = 1$.
The first assertion of Theorem \ref{est-struc-modif-st}, applied to the convex problems with inexact oracle model (recall the example (iii) in Section \ref{sssec-examples-struc} and the fact that $\sigma_f = \bar{\sigma}_f$), is slightly better than the estimate \cite[Theorem 7]{DGNs} in view of $(L-\sigma_f)l_d(\z_k;x^*)\leq Ld(x^*)$ and $\frac{L-\sigma_f}{\sigma_f} \leq \frac{L}{\sigma_f}$.
Furthermore, the first assertion applied to the composite problems $\min_{x \in Q}[f(x)\equiv f_0(x)+\varPsi(x)]$ (the example (ii) in Section \ref{sssec-examples-struc}) is the same as Nesterov's one \cite[Theorem 6]{Nes13} with $\gamma_u = 2$ (recall that $\bar\sigma_f = \sigma_{f_0}=0, \sigma_f=\sigma_{\varPsi}$).
Therefore, Method \ref{gen-alg-struc} achieves the optimal complexity for smooth and strongly convex problems (see Section \ref{ssec-existing-struc}). 

The second result of Theorem \ref{est-struc-modif-st} matches the conclusion for the classical CGM observed in \cite[Section 5.2.1]{FG14}. If we further assume  $f \in C_L^{1,1}(Q)$, then the corresponding implementation of the second assertion with the extended MD model (\ref{eMD-model}) and the DA model (\ref{DA-model}) yield particular instances of the CGMs proposed by Lan \cite{Lan14a} (see Section \ref{sssec-existing-struc}).


\subsection{Optimal convergence rates of the modified method for weakly smooth problems}
\label{ssec-eff-modif-weak}

Considering structured problems in the case when $\delta(y,x)=\frac{M(y)}{\rho}\norm{y-x}^{\rho},~\rho \in [1,2)$, we can provide convergence analysis for problems involving weakly smooth functions of the class $C_{M}^{1,\rho-1}(Q)$ (see examples (iv) and (v) in Section \ref{sssec-examples-struc}).
Note that the smooth case $\rho=2$ reduces to the situation $\delta(y,x)=0$ which has been already discussed.
In this section, we show convergence results of modified proximal/conditional gradient methods for this setting.
In the case $\rho=1$, the results from Sections \ref{sssec-general-bound} to \ref{sssec-convergence-st} can be seen as variants of stochastic gradient methods developed in \cite{CLP12,GL12} for the deterministic setting.


\subsubsection{General convergence estimates of the modified method for weakly smooth problems}
\label{sssec-general-bound}

Our analysis for proximal gradient methods is based on the following lemma.

\begin{lem}\label{Holder-bound-lemma}
Consider a structured problem in the class $\mathcal{SP}(m_f,\sigma_f,\bar\sigma_f,L,\delta)$.
Assume that $\delta(y,x)=\frac{M(y)}{\rho}\norm{y-x}^\rho,~\rho \in [1,2)$, $M(\cdot)\geq 0$.
Let $\{(\z_{k-1},\w_{k-1},x_k,\hat{x}_k)\}_{k \geq 0}$ be generated by the modified method of Method \ref{gen-alg-struc} with weight parameters $\{\lambda_k\}_{k\geq 0}$ and scaling parameters $\{\beta_k\}_{k\geq -1}$. Put $\alpha_k := L(x_k)-\sigma_d\left(\bar\sigma_f + \frac{S_{k}(\beta_{k-1}+S_{k-1}\sigma_f)}{\lambda_k^2}\right)$. If $\alpha_i<0$ for each $0\leq i \leq k$, then we have 
\[ f(\hat{x}_k)-f(x^*) +\sigma_f\xi(\z_k,x^*) \leq \frac{\beta_k l_d(\z_k;x^*)}{S_k}+\frac{(2-\rho)\max_{0\leq i \leq k}M(x_i)^{\frac{2}{2-\rho}}}{2\rho S_k}\sum_{i=0}^k\frac{S_i}{(-\alpha_i)^{\frac{\rho}{2-\rho}}}.
\]
\end{lem}

\begin{proof}
Note that the function $g(r)=ar^2+br^{\rho}$ for $r \geq 0, a<0, b \in \Real$ satisfies $\max_{r \geq 0}g(r)=\frac{2-\rho}{2\rho}(-2a)^{\frac{-\rho}{2-\rho}}(\rho b)^{\frac{2}{2-\rho}}$.
Hence, Theorem \ref{gen-conv-struc} concludes that
\begin{eqnarray*}
f(\hat{x}_k)-f(x^*) +\sigma_f\xi(\z_k,x^*)
&\leq&
\frac{\beta_k l_d(\z_k;x^*)}{S_k}+\frac{1}{S_k}\sum_{i=0}^kS_i\left(\frac{1}{2}\alpha_i\norm{\hat{x}_i-x_i}^2+\frac{M(x_i)}{\rho}\norm{\hat{x}_i-x_i}^{\rho}\right)\\
&\leq&
\frac{\beta_k l_d(\z_k;x^*)}{S_k}+\frac{1}{S_k}\sum_{i=0}^kS_i\times\frac{2-\rho}{2\rho}(-\alpha_i)^{\frac{-\rho}{2-\rho}}M(x_i)^{\frac{2}{2-\rho}},
\end{eqnarray*}
which proves the assertion.
\end{proof}

\subsubsection{Optimal convergence rates for the non-strongly convex case}
\label{sssec-convergence-nonst}

Let us deduce a convergence result of PGMs given by the modified method of Method \ref{gen-alg-struc} for the non-strongly convex case $\sigma_f=\bar{\sigma}_f=0$.
The result with $\rho=1$ is closely related to the deterministic versions of \cite[Proposition 8]{GL12} and \cite[Corollary 1]{CLP12}.

\begin{thm}\label{Holder-nonst}
Consider a structured problem in the class $\mathcal{SP}(m_f,\sigma_f,\bar\sigma_f,L,\delta)$.
Assume additionally that $L(\cdot)=L\geq 0$, $\sigma_f=\bar\sigma_f=0$, and $\delta(y,x)=\frac{M(y)}{\rho}\norm{y-x}^\rho$ for $\rho \in [1,2)$, $M(\cdot)\geq 0$.
Let $\{(\z_{k-1},\w_{k-1},x_k,\hat{x}_k)\}_{k \geq 0}$ be generated by the modified method of Method \ref{gen-alg-struc} with
\[ \lambda_k:=\frac{k+1}{2},\quad \beta_{k}:=\frac{L}{\sigma_d}+\frac{\gamma}{\sigma_d}(k+3)^{\frac{3}{2}(2-\rho)},\quad \gamma > 0. \]
Then, for every $k \geq 0$, we have
\[ f(\hat{x}_k)-f(x^*) \leq
\frac{4Ll_d(\z_k;x^*)}{\sigma_d(k+1)(k+2)}+
\left[\frac{4\gamma l_d(\z_k;x^*)}{\sigma_d}
+ \frac{\max_{0\leq i \leq k}M(x_i)^{\frac{2}{2-\rho}}}{3\rho\gamma^{\frac{\rho}{2-\rho}}}\right]\frac{(k+3)^{\frac{3}{2}(2-\rho)}}{(k+1)(k+2)}. \]
\end{thm}

\begin{proof}
We apply Lemma \ref{Holder-bound-lemma} to prove the assertion. Note that
\begin{equation}\label{Holder-nonst-item-1}
\frac{\beta_k}{S_k}=\frac{4L}{\sigma_d(k+1)(k+2)}+\frac{4\gamma(k+3)^{\frac{3}{2}(2-\rho)}}{\sigma_d(k+1)(k+2)}
\end{equation}
and $\alpha_k$ in Lemma \ref{Holder-bound-lemma} becomes now
$\alpha_k = -\frac{L}{k+1} -\gamma\frac{(k+2)^{\frac{3}{2}(2-\rho)+1}}{k+1} \leq -\gamma\frac{(k+2)^{\frac{3}{2}(2-\rho)+1}}{k+1} <0$.
Furthermore, we have
\begin{eqnarray}
\frac{1}{S_k}\sum_{i=0}^k\frac{S_i}{(-\alpha_i)^{\frac{\rho}{2-\rho}}}
&\leq&
\frac{1}{S_k}\sum_{i=0}^k\frac{(i+1)^{\frac{\rho}{2-\rho}+1}}{4\gamma^{\frac{\rho}{2-\rho}}(i+2)^{\frac{3}{2}\rho+\frac{\rho}{2-\rho}-1}}
\leq
\frac{1}{4\gamma^{\frac{\rho}{2-\rho}}S_k}\sum_{i=0}^k(i+2)^{2-\frac{3}{2}\rho}\nonumber\\
&\leq&
\frac{1}{4\gamma^{\frac{\rho}{2-\rho}}S_k}\frac{2}{3(2-\rho)}(k+3)^{3-\frac{3}{2}\rho}
= \frac{2(k+3)^{\frac{3}{2}(2-\rho)}}{3(2-\rho)\gamma^{\frac{\rho}{2-\rho}}(k+1)(k+2)},
\label{Holder-nonst-item-2}
\end{eqnarray}
where the second and the third inequalities are due to $i+1 \leq i+2$ and the fact
$ \sum_{i=0}^k(i+2)^q \leq \frac{1}{1+q}(k+3)^{1+q},~\forall q > -1$, respectively.
Consequently, the theorem follows by applying Lemma \ref{Holder-bound-lemma} with the inequalities (\ref{Holder-nonst-item-1}) and (\ref{Holder-nonst-item-2}).
\end{proof}

Notice that we need the parameter $\rho$ to define $\beta_k$ but not the $M(\cdot)$. Now let us calculate an efficient choice for $\gamma$.
Suppose that $M(\cdot)\leq\hat{M}<+\infty$.
Using $l_d(z_k;x^*) \leq d(x^*)$ and the fact that the function $g(\gamma)=a\gamma+\frac{b}{\gamma^p}~ (a,b,p>0)$ attains its minimum at $\gamma^*=(pb/a)^{\frac{1}{p+1}}$ on $(0,\infty)$ with $g(\gamma^*)=(p+1)p^{\frac{-p}{p+1}}a^{\frac{p}{p+1}}b^{\frac{1}{p+1}}$,
the choice
\[
\gamma=\gamma^*:=
\left(\frac{\rho}{2-\rho}\frac{\hat{M}^{\frac{2}{2-\rho}}}{3\rho}\frac{\sigma_d}{4d(x^*)}\right)^{\frac{2-\rho}{2}}
=
\hat{M}\left(\frac{\sigma_d}{12(2-\rho)d(x^*)}\right)^{\frac{2-\rho}{2}}\]
makes the estimate of Theorem \ref{Holder-nonst} as follows:
\begin{eqnarray*}
f(\hat{x}_k)-f(x^*)
&\leq&
\frac{4Ld(x^*)}{\sigma_d(k+1)(k+2)}+
\frac{2}{2-\rho}\left(\frac{\rho}{2-\rho}\right)^{-\frac{\rho}{2}}\left(\frac{4d(x^*)}{\sigma_d}\right)^{\frac{\rho}{2}}\left(\frac{\hat{M}^{\frac{2}{2-\rho}}}{3\rho}\right)^{\frac{2-\rho}{2}}\frac{(k+3)^{\frac{3}{2}(2-\rho)}}{(k+1)(k+2)}\\
&=&
\frac{4Ld(x^*)}{\sigma_d(k+1)(k+2)}+
\frac{2(2\sqrt{3})^{\rho}}{3\rho(2-\rho)^{\frac{2-\rho}{2}}}\hat{M}\left(\frac{d(x^*)}{\sigma_d}\right)^{\frac{\rho}{2}}\frac{(k+3)^{\frac{3}{2}(2-\rho)}}{(k+1)(k+2)}.
\end{eqnarray*}
Note that $\min_{x > 0}x^x = (1/e)^{1/e}$ and $\max_{\rho\in[1,2]}\frac{2}{3\rho}(2\sqrt{3})^\rho=\frac{2}{3\cdot 2}(2\sqrt{3})^2=4$ because $\log(2\sqrt{3})>1$ implies the positivity of the derivative of $\frac{2}{3\rho}(2\sqrt{3})^\rho$.
Therefore, we have $\frac{2(2\sqrt{3})^{\rho}}{3\rho(2-\rho)^{\frac{2-\rho}{2}}} \leq 4e^{1/(2e)}$
which shows $f(\hat{x}_k)-f(x^*) \leq O\left(\frac{Ld(x^*)}{\sigma_d}k^{-2}+\hat{M}\left(\frac{d(x^*)}{\sigma_d}\right)^{\frac{\rho}{2}}k^{-\frac{3\rho-2}{2}} \right)$.
Consequently, we obtain an upper bound of the iteration complexity to obtain $f(\hat{x}_k)-f(x^*)\leq \varepsilon$ which is proportional to
\[
\left(\frac{Ld(x^*)}{\sigma_d\varepsilon}\right)^{\frac{1}{2}}+\left(\frac{d(x^*)}{\sigma_d}\right)^{\frac{\rho}{3\rho-2}}\left(\frac{\hat{M}}{\varepsilon}\right)^{\frac{2}{3\rho-2}}.
\]
In view of the lower complexity (\ref{lower-compl-weak}) (with $L$ replaced by $\hat{M}$ there), it turns out that the order of the second term is optimal for the class $C_{\hat{M}}^{1,\rho-1}(Q)$.

\subsubsection{Optimal convergence rate for the strongly convex case}
\label{sssec-convergence-st}

Now we show a convergence result of PGMs for the strongly convex case $\sigma_f > 0$.

\begin{thm} 
Consider a structured problem in the class $\mathcal{SP}(m_f,\sigma_f,\bar\sigma_f,L,\delta)$.
Assume additionally that $L(\cdot)=L\geq 0$, $\sigma_f>0$, and $\delta(y,x)=\frac{M(y)}{\rho}\norm{y-x}^\rho$ for $\rho \in [1,2)$, $M(\cdot)\geq 0$.
Let $\{(\z_{k-1},\w_{k-1},x_k,\hat{x}_k)\}_{k \geq 0}$ be generated by the modified method of Method \ref{gen-alg-struc} with
\[ \lambda_k := \frac{1}{p+1}(k+1)^{p},\quad \beta_{k} := \left(\frac{L}{\sigma_d}+\beta\right)(k+2)^{p-1}  \]
where $p\geq1$ and $\beta \geq 0$ with $\sigma_d\bar\sigma_f+pL+(p+1)\sigma_d\beta > 0$.
Then, for every $k \geq 0$, we have
\begin{eqnarray*}
f(\hat{x}_k)-f(x^*) +\sigma_f\xi(z_k, x^*) &\leq&
\left(\frac{L}{\sigma_d}+\beta\right)(p+1)^{2}l_d(\z_k;x^*)\frac{(k+2)^{p-1}}{(k+1)^{p+1}}\\
&& + \frac{(p+1)(2-\rho)\max_{0\leq i\leq k}M(x_i)^{\frac{2}{2-\rho}}}{2\rho(\sigma_d\bar\sigma_f+pL+(p+1)\sigma_d\beta)^{\frac{\rho}{2-\rho}}}\frac{1}{(k+1)^{p+1}} \\
&& + \frac{3^{p+1}(2-\rho)\max_{0\leq i\leq k}M(x_i)^{\frac{2}{2-\rho}}}{2\rho} \left(\frac{2^{p-1}(p+1)^2}{\sigma_d\sigma_f}\right)^{\frac{\rho}{2-\rho}}  P(k),
\end{eqnarray*}
where
\[
P(k)=\left\{
\begin{array}{ll}
\left(p+2-\frac{2\rho}{2-\rho}\right)^{-1}(k+1)^{-\frac{3\rho-2}{2-\rho}} & : p+1 > \frac{3\rho-2}{2-\rho},\\[2mm]
\ds\frac{1+\log k}{(k+1)^{p+1}} & : p+1 = \frac{3\rho-2}{2-\rho},\\
\ds\frac{1 - \left(p+2-\frac{2\rho}{2-\rho}\right)^{-1}}{(k+1)^{p+1}} & : p+1 < \frac{3\rho-2}{2-\rho}.
\end{array}
\right.
\]
\end{thm}

\begin{proof}
Note that $\beta_k$ is non-decreasing and $\frac{1}{(p+1)^2}(k+1)^{p+1}\leq S_k \leq  \frac{1}{(p+1)^2}(k+2)^{p+1}$.
Then, we have
\begin{equation}\label{Holder-st-1}
\frac{\beta_k}{S_k} \leq \left(\frac{L}{\sigma_d}+\beta\right)(p+1)^{2}\frac{(k+2)^{p-1}}{(k+1)^{p+1}}=O(k^{-2}).
\end{equation}
Since the inequalities $\frac{S_k}{\lambda_k^2} \geq \frac{1}{(k+1)^{p-1}}$ and $\frac{S_kS_{k-1}}{\lambda_k^2} \geq \frac{1}{(p+1)^2}\frac{k^{p+1}}{(k+1)^{p-1}} \geq \frac{k^2}{2^{p-1}(p+1)^2}$ for $k\geq 1$ imply
\[
-\alpha_k := \sigma_d\left(\bar\sigma_f+\frac{S_k(\beta_{k-1}+S_{k-1}\sigma_f)}{\lambda_k^2}\right)-L \geq \sigma_d\bar\sigma_f + \beta\sigma_d + \frac{\sigma_d\sigma_f}{2^{p-1}(p+1)^2}k^2 > 0,~~~ k\geq 1,
\]
we obtain 
\[
\frac{S_k}{(-\alpha_k)^{\frac{\rho}{2-\rho}}}
< \frac{1}{(p+1)^2}\left(\frac{2^{p-1}(p+1)^2}{\sigma_d\sigma_f}\right)^{\frac{\rho}{2-\rho}}\frac{(k+2)^{p+1}}{k^{\frac{2\rho}{2-\rho}}}
\leq \frac{3^{p+1}}{(p+1)^2}\left(\frac{2^{p-1}(p+1)^2}{\sigma_d\sigma_f}\right)^{\frac{\rho}{2-\rho}} k^{p+1-\frac{2\rho}{2-\rho}}
\]
for all $k \geq 1$.
Combining with $\frac{S_0}{(-\alpha_0)^{\frac{\rho}{2-\rho}}} = \frac{1}{(p+1)(\sigma_d\bar\sigma_f+pL+(p+1)\sigma_d\beta)^{\frac{\rho}{2-\rho}}}$ yields that
\begin{equation}\label{Holder-st-2}
\frac{1}{S_k}\sum_{i=0}^k\frac{S_i}{(-\alpha_i)^{\frac{\rho}{2-\rho}}}
\leq
\frac{p+1}{(\sigma_d\bar\sigma_f+pL+(p+1)\sigma_d\beta)^{\frac{\rho}{2-\rho}}}\frac{1}{(k+1)^{p+1}} + 3^{p+1} \left(\frac{2^{p-1}(p+1)^2}{\sigma_d\sigma_f}\right)^{\frac{\rho}{2-\rho}}  P(k),
\end{equation}
where the factor $P(k)$ is due to the following inequality:
\[
\sum_{i=1}^ki^q \leq \left\{
\begin{array}{ll}
\frac{1}{1+q}(k+1)^{q+1} & : q >-1,\\
1+\log k & : q = -1,\\
1- \frac{1}{1+q} & : q < -1.
\end{array}
\right.
\]
Consequently, the assertion follows from Lemma \ref{Holder-bound-lemma} with the inequalities (\ref{Holder-st-1}) and (\ref{Holder-st-2}).
\end{proof}
Notice that we do not need $\rho$ and $M(\cdot)$ in the definition of the parameters $\lambda_k,\beta_k$; the result holds for all acceptable $\rho \in [1,2)$.
If we further have $p+1>\frac{3\rho-2}{2-\rho}$, then $P(k)$ has the best rate of convergence for a fixed $\rho$.
Now let us see the above upper bound in the case $L=0,~\sigma_f=\bar\sigma_f>0,~M(\cdot)=M,~\beta=0,~p+1>\frac{3\rho-2}{2-\rho}$:
\[
\begin{array}{lcl}
f(\hat{x}_k)-f(x^*)+\sigma_f\xi(\z_k,x^*)
&\leq&
\frac{(p+1)(2-\rho)M^{\frac{2}{2-\rho}}}{2\rho(\sigma_d\sigma_f)^{\frac{\rho}{2-\rho}}}\frac{1}{(k+1)^{p+1}}
\\&&
+
\frac{3^{p+1}(2-\rho)}{2\rho}M^{\frac{2}{2-\rho}} \left(\frac{2^{p-1}(p+1)^2}{\sigma_d\sigma_f}\right)^{\frac{\rho}{2-\rho}}\left(p+2-\frac{2\rho}{2-\rho}\right)^{-1}(k+1)^{-\frac{3\rho-2}{2-\rho}}.
\end{array}
\]
Since this bound is of $O\left(c(p,\rho)\frac{M^{2/(2-\rho)}}{(\sigma_d\sigma_f)^{\rho/(2-\rho)}}k^{-\frac{3\rho-2}{2-\rho}}\right)$ for a continuous function $c(p,\rho)$, it achieves the optimal complexity (\ref{lower-compl-weak}) for the strongly convex case.
In contrast to the optimal method in \cite{NN85}, we do not need to restart the method and do not require $M$ and an upper bound of $d(x^*)$ in advance%
\footnote{As is indicated in \cite{NN85}, an obvious upper bound of $d(x^*)$ can be obtained if $\nabla{f}(x^*)=0$ and we know $M$ for the weakly smooth problems (example (iv) in Section \ref{sssec-examples-struc}) in the Euclidean setting $d(x)=\frac{1}{2}\norm{x-x_0}_2^2$~: The inequality $d(x^*) \leq \frac{1}{2}(\frac{2M}{\rho\sigma_f})^{2/(2-\rho)}$ follows since we have $\frac{\sigma_f}{2}\norm{x^*-x_0}_2^2 \leq f(x_0)-f(x^*) \leq \frac{M}{\rho}\norm{x_0-x^*}_2^\rho$ (recall the strong convexity and (\ref{struc-ineq})).}%
~to ensure the optimality.

Let us consider the non-smooth case $\rho=1,~\bar\sigma_f=\sigma_f>0$.
Then, taking $p=1$ and $\beta=0$ yields $\lambda_k=(k+1)/2,~\beta_{k-1}=L/\sigma_d$, and
\[ f(\hat{x}_k)-f(x^*)+\sigma_f\xi(z_k,x^*) \leq\frac{4Ll_d(\z_k;x^*)}{\sigma_d(k+1)^2} + \frac{\max_{0\leq i\leq k}M(x_i)^2}{(\sigma_d\sigma_f+L) (k+1)^2} + \frac{18\max_{0\leq i\leq k}M(x_i)^2}{\sigma_d\sigma_f (k+1)}. \]
This result is similar to the ones \cite[Proposition 9]{GL12} and \cite[Corollary 2]{CLP12} in the deterministic case.

\subsubsection{Optimal/nearly optimal convergence rate of conditional gradient methods}
\label{sssec-CGM-weak}
We finally consider the case of conditional gradient methods: $\beta_k \equiv 0,~\sigma_f=\bar\sigma_f=0$. This case can be analyzed without Lemma \ref{Holder-bound-lemma}.

\begin{thm}\label{Holder-CG}
Consider a structured problem in the class $\mathcal{SP}(m_f,\sigma_f,\bar\sigma_f,L,\delta)$.
Assume additionally that $L(\cdot)=L\geq 0$, $\sigma_f=\bar\sigma_f=0$, and $\delta(y,x)=\frac{M}{\rho}\norm{y-x}^\rho$ for $\rho \in [1,2)$, $M \geq 0$.
Then, the modified method of Method \ref{gen-alg-struc} for the problem with $\lambda_k=(k+1)/2$ and $\beta_k \equiv 0$ generates a sequence $\{\hat{x}_k\}_{k \geq 0} \subset Q$ satisfying
\begin{equation}\label{Holder-CG-bound}
f(\hat{x}_k)-f(x^*) \leq \frac{2L{\rm Diam}(Q)^2}{k+4} + \frac{2^{\rho+1} M{\rm Diam}(Q)^\rho}{\rho(3-\rho)(k+2)^{\rho-1}}
\end{equation}
for every $k \geq 0$.
\end{thm}

\begin{proof}
Theorem \ref{gen-conv-struc} yields that $f(\hat{x}_k)-f(x^*)\leq C_k/S_k$ with $S_k = (k+1)(k+2)/4$ and
\[ C_k = \sum_{i=0}^kS_i\left(\frac{L}{2}\norm{\hat{x}_i-x_i}^2+\frac{M}{\rho}\norm{\hat{x}_i-x_i}^\rho\right) = \sum_{i=0}^k\left(\frac{L}{2}\frac{\lambda_i^2}{S_i}\norm{\w_i-\z_{i-1}}^2+\frac{M}{\rho}\frac{\lambda_i^\rho}{S_i^{\rho-1}}\norm{\w_i-\z_{i-1}}^\rho\right) \]
(see Remark \ref{CGM-gen-bound}). Using the inequality (\ref{FG-prop}) and
\[
\sum_{i=0}^k\frac{\lambda_i^\rho}{S_i^{\rho-1}}=\frac{1}{2^{2-\rho}}\sum_{i=0}^k\frac{i+1}{(i+2)^{\rho-1}} \leq \frac{1}{2^{2-\rho}}\sum_{i=0}^k(i+1)^{2-\rho} \leq \frac{1}{2^{2-\rho}(3-\rho)}(k+2)^{3-\rho}\]
(the first and the second inequalities are due to $i+1\leq i+2$ and the fact $\sum_{i=0}^k(i+1)^q\leq\frac{1}{1+q}(k+2)^{1+q}$ for $q \geq 0$, respectively),
we conclude that
\[ f(\hat{x}_k)-f(x^*) \leq \frac{C_k}{S_k}\leq \frac{2L {\rm Diam(Q)}^2}{k+4} + \frac{2^\rho M {\rm Diam(Q)}^\rho}{\rho(3-\rho)}\frac{(k+2)^{2-\rho}}{k+1}. \]
The estimate (\ref{Holder-CG-bound}) now follows from $\frac{k+2}{k+1} \leq 2$ for $k \geq 0$.
\end{proof}
The bound (\ref{Holder-CG-bound}) is also valid for the classical CGM (\ref{classical-CGM}) with $\tau_k:=\lambda_{k+1}/S_{k+1}=\frac{2}{k+3},~\hat{x}_k:=x_k$; it can be derived in the same way as Theorem \ref{Holder-CG} based on the estimate (\ref{clas-CGM-bound}) since $f(x_0)-m_f(x_0;\sz_0) \stackrel{(\ref{struc-ineq})}{\leq} \frac{L}{2}{\rm Diam}(Q)^2 + \frac{M}{\rho}{\rm Diam}(Q)^{\rho}$ and $\delta(x_{k-1},x_k)=\frac{M}{\rho}\norm{x_{k}-x_{k-1}}^\rho \stackrel{(\ref{classical-CGM})}{=} \frac{M}{\rho}\frac{\lambda_k^\rho}{S_k^\rho}\norm{x_{k-1}-\sz_{k-1}}^\rho\leq\frac{M}{\rho}\frac{\lambda_k^\rho}{S_k^\rho}{\rm Diam}(Q)^\rho$ for $k \geq 1$. This result in the case $L=0$ is very similar to a known result for the classical CGM (see \cite[Proposition 1.1]{CJN13} and \cite{Nes15}).

Since the choice $\lambda_k=(k+1)/2$ and $\beta_k\equiv0$ are independent of $L, M$, and $\rho$, the conditional gradient methods can be applied to the classes $C_{M}^{1,\rho-1}(Q), ~\rho\in(1,2]$ ensuring the convergence $f(\hat{x}_k)-f(x^*)\leq O\left(\frac{M{\rm Diam(Q)}^{\rho}}{k^{\rho-1}}\right)$.
Thus, our CGMs ensure the same convergence rate as the known one (\ref{CGM-known-rate}) of existing CGMs for weakly smooth problems.

When we choose the extended MD model (\ref{eMD-model}) or the DA model (\ref{DA-model}) in Theorem \ref{Holder-CG}, the obtained CGMs match particular cases of Lan's CGMs mentioned in Section \ref{sssec-existing-struc}. Since the convergence rates for Lan's CGMs was analyzed only for smooth problems in \cite{Lan14a}, our result provides a generalization of them for weakly smooth problems.

\section{Conclusion}
This paper proposes a new framework for applying (sub)gradient-based methods to minimize strongly convex functions. It unifies the analysis of PGMs and CGMs for several classes of problems including non-smooth, smooth, and weakly smooth problems.
We have introduced the notion of strong convexity with respect to the prox-function, which generalizes the one in the Euclidean setting.
The proposed PGMs establish optimal convergence rates for these problems with slight improvements than some existing methods. Furthermore, particular cases of the framework yield a family of variations of the classical CGM with optimal and nearly optimal guarantee of convergence in the non-strongly convex case.

A remarkable novel result in this paper, in view of method efficiency, is the achievement of the optimal complexity for the weakly smooth problems (the class $C_{M}^{1,\nu}(Q),~\nu\in[0,1)$) in the strongly convex case without knowing the constant $M$ and an upper bound of $d(x^*)$ (Section \ref{sssec-convergence-st}; see also Section \ref{sssec-examples-struc} (iv) for remarks on the literature).
The theoretical approach for that is similar to the ones in \cite{DGNs,DGN,Nes15u} because the structure (\ref{struc-ineq}) assumes an oracle inexactness of the original problem. Furthermore, the analysis of Sections \ref{sssec-convergence-nonst} and \ref{sssec-convergence-st} can be seen as a generalization of the techniques of \cite{GL12,GL13} in the deterministic case.

We finally describe several topics for further considerations.
At first, we can consider a generalization/combination of the (sub)gradient-based methods here with smoothing techniques, stochastic situations, or uniformly convex settings. Related studies can be seen in \cite{GL12,GL13,JN14,Lan14a}.
Secondly, one can further consider to tune the parameters, the weight and the scaling ones, to obtain an efficient convergence. The proposed choices in Section \ref{sec-convergence} are not the only way to ensure the optimal convergence; see, \eg, \cite{FG14,NB} for some discussions on other choices.
Thirdly, it is important to note that the convergence results for the class $C_{M}^{1,\nu}(Q)$ in Sections \ref{sssec-convergence-nonst}, \ref{sssec-convergence-st} are not {\it adaptive} in contrast to the known method \cite{Nes15u} proposed by Nesterov; namely , it does not ensure the optimal convergence without knowing the parameter $\nu$. From the practical viewpoint, it will be important to develop techniques to ensure efficient convergence rates without such problem specific information.

\section*{Acknowledgements}
The author is very thankful to the anonymous referees who gave constructive suggestions which improved substantially the readability of the paper. He is also thankful to Prof. Mituhiro Fukuda for comments and suggestions and also to Prof. Guanghui Lan for pointing out some related results. This work was partially supported by JSPS Grant-in-Aid for Scientific Research (C) number 26330024.


\appendix
\section{Appendix}

{\normalsize 

In order to complete the proof of Theorem \ref{est-struc-modif-st}, we need to obtain upper bounds for $1/S_k$ and $\sum_{i=0}^kS_i/S_k$ for the sequence $\{S_k\}_{k \geq 0}$ defined by (\ref{param-struc-modif}).
Since $\lambda_{k+1}=S_{k+1}-S_k$, writing
$r:=\frac{\sigma_f\sigma_d}{L-\bar\sigma_f\sigma_d}\geq 0$, the sequence $\{S_k\}_{k\geq 0}$ in (\ref{param-struc-modif}) is determined by the recurrence
\begin{equation}\label{S_k-def}
S_0=1,\quad (S_{k+1}-S_k)^2=S_{k+1}(1+rS_k),\quad k \geq 0
\end{equation}
where the root of the equation in $S_{k+1}$ takes the largest one, namely,
\begin{equation}\label{S_k-rec}
S_{k+1}=\frac{1+(2+r)S_k+\sqrt{(1+(2+r)S_k)^2-4S_k^2}}{2}.
\end{equation}
The essentials of lemmas below are the same as \cite[Lemma 4-7]{DGNs} excepting the replacement of $\mu/L$ in the article by an arbitrary $r\geq 0$.

\begin{applem}\label{lem-S_k-lower-bound-b}
For any sequence $\{S_k\}_{k \geq 0}$ defined by (\ref{S_k-def}) for $r \geq 0$, we have
\[ \frac{1}{S_k} \leq \min\left\{\frac{4}{(k+1)(k+4)}, \left(\frac{2}{2+r+\sqrt{r^2+4r}}\right)^{k}\right\},\quad \forall k \geq 0. \]
\end{applem}

\begin{proof}
Since $S_{k+1}\geq S_k$, we have
\begin{equation}
\sqrt{S_{k+1}}-\sqrt{S_{k}} = \frac{S_{k+1}-S_{k}}{\sqrt{S_{k+1}}+\sqrt{S_{k}}} \geq \frac{S_{k+1}-S_{k}}{2\sqrt{S_{k+1}}} \stackrel{(\ref{S_k-def})}{=} \frac{1}{2}\sqrt{1+rS_k}\geq \frac{1}{2}
\end{equation}
which shows $\sqrt{S_k}\geq \frac{k}{2}+\sqrt{S_0}=\frac{k+2}{2}$ for all $k \geq 0$. Then, we have
\[ S_{k}-S_0=\sum_{i=0}^{k-1}(S_{i+1}-S_i)\stackrel{(\ref{S_k-def})}{=}\sum_{i=0}^{k-1}\sqrt{S_{i+1}(1+rS_i)}\geq \sum_{i=0}^{k-1}\sqrt{S_{i+1}}\geq \sum_{i=0}^{k-1}\frac{i+3}{2}=\frac{k(k+5)}{4} \]
which gives $S_k \geq S_0+\frac{k(k+5)}{4}=\frac{(k+1)(k+4)}{4}$.
On the other hand, using (\ref{S_k-rec}) yields that
\begin{equation}\label{S_k-factor}
\frac{S_{k+1}}{S_k}=\frac{\frac{1}{S_k}+2+r+\sqrt{\left(\frac{1}{S_k}+(2+r)\right)^2-4}}{2} \geq \frac{2+r+\sqrt{(2+r)^2-4}}{2} = \frac{2+r+\sqrt{r^2+4r}}{2}
\end{equation}
for all $k \geq 0$.
Hence, we have $S_k\geq S_0\left(\frac{2+r+\sqrt{r^2+4r}}{2}\right)^k=\left(\frac{2+r+\sqrt{r^2+4r}}{2}\right)^k$.
\end{proof}

\begin{rem*}
The linear convergence factor $\frac{2}{2+r+\sqrt{r^2+4r}}$ in the above lemma satisfies
\[1-\sqrt{\frac{r}{r+1}}  \leq \frac{2}{2+r+\sqrt{r^2+4r}} \leq \left(1+\frac{1}{2}\sqrt{r}\right)^{-2}.\]
In fact, since
\[ \left(1-\sqrt{\frac{r}{r+1}}\right)^{-1}=\frac{\sqrt{r+1}}{\sqrt{r+1}-\sqrt{r}}=\sqrt{r+1}(\sqrt{r+1}+\sqrt{r})=\frac{2+2r+\sqrt{4r^2+4r}}{2}, \]
we obtain
\[
\left(1+\frac{1}{2}\sqrt{r}\right)^{2} = \frac{2+r/2+\sqrt{4r}}{2} \leq \frac{2+r+\sqrt{r^2+4r}}{2} \leq \frac{2+2r+\sqrt{4r^2+4r}}{2}=\left(1-\sqrt{\frac{r}{r+1}}\right)^{-1}.
\]
Note that if $\bar{\sigma}_f=\sigma_f$ and $r=\frac{\sigma_f\sigma_d}{L-\bar\sigma_f\sigma_d}$, then $\sqrt{\frac{r}{r+1}}=\sqrt{\frac{\sigma_f\sigma_d}{L}}$.
\end{rem*}

\begin{applem}\label{lem-S_k-bound-st}
The sequence $\{S_k\}_{k \geq 0}$ defined by (\ref{S_k-def}) for $r>0$ satisfies
\[ \frac{\sum_{i=0}^kS_i}{S_k} \leq \frac{1+\sqrt{1+4r^{-1}}}{2}\leq 1+\sqrt{\frac{1}{r}},\quad \forall k \geq 0. \]
\end{applem}

\begin{proof}
Notice that $\gamma := \frac{1+\sqrt{1+4r^{-1}}}{2}$ satisfies
\[ \left(1-\frac{1}{\gamma}\right)^{-1}=\frac{\gamma}{\gamma-1}=\frac{\sqrt{1+4r^{-1}}+1}{\sqrt{1+4r^{-1}}-1}=\frac{(\sqrt{1+4r^{-1}}+1)^2}{4r^{-1}}=\frac{2+r+\sqrt{r^2+4r}}{2}. \]
Therefore, we obtain $\frac{S_k}{S_{k+1}}\leq 1-\frac{1}{\gamma}$ by (\ref{S_k-factor}). Now the result follows by induction: If $\sum_{i=0}^kS_i/S_k\leq \gamma$ holds for some $k \geq 0$, we have
\[ \frac{\sum_{i=0}^{k+1}S_i}{S_{k+1}} = 1+\frac{S_k}{S_{k+1}}\frac{\sum_{i=0}^kS_i}{S_k} \leq 1+\frac{\gamma-1}{\gamma}\cdot\gamma=\gamma. \]
This proves the first inequality; the second can be verified from $\sqrt{1+4r^{-1}}\leq 1+2\sqrt{r^{-1}}$.
\end{proof}

Note that the result of Lemma \ref{lem-S_k-bound-st} is the same as \cite[Lemma 5]{DGNs} because $1+\frac{2\sqrt{r^{-1}}}{\sqrt{r}+\sqrt{r+4}}=\frac{1+\sqrt{1+4r^{-1}}}{2}$.

\begin{applem} 
Let $\{S_k\}_{k \geq 0}$ be defined as Lemma \ref{lem-S_k-bound-st} and  $\{T_k\}_{k \geq 0}$ be defined by (\ref{S_k-def}) with $r:=0$, namely $T_0:=1$ and $T_{k+1}:=\frac{1+2T_k+\sqrt{1+4T_k}}{2}$ for $k \geq 0$. Then, we have
\begin{equation*}
\frac{\sum_{i=0}^kS_i}{S_k} \leq \frac{\sum_{i=0}^kT_i}{T_k},\quad \forall k \geq 0.
\end{equation*}
\end{applem}

\begin{proof}
Due to the identity
\[ \frac{\sum_{i=0}^kS_i}{S_k} = 1+\sum_{i=0}^{k-1}\frac{S_i}{S_{k}} =  1+\sum_{i=0}^{k-1}\prod_{j=i}^{k-1}\frac{S_j}{S_{j+1}}, \quad k \geq 0, \]
it is enough to show that $\frac{S_k}{S_{k+1}} \leq \frac{T_k}{T_{k+1}}$ for every $k \geq 0$. Notice that we have
\begin{equation}\label{S_k-T_k-ratio}
\frac{S_{k+1}}{S_k} = \frac{\frac{1+rS_k}{S_k}+2+\sqrt{\left(\frac{1+rS_k}{S_k}+2\right)^2-4}}{2},\quad \frac{T_{k+1}}{T_k} = \frac{\frac{1}{T_k}+2+\sqrt{\left(\frac{1}{T_k}+2\right)^2-4}}{2},
\end{equation}
which suggests us to prove $\frac{1+rS_k}{S_k} \geq \frac{1}{T_k}$ for $k \geq 0$. It is true for $k=0$ by $S_0=T_0$. If it holds for $k \geq 0$, then, writing $\alpha:=\frac{1+rS_k}{S_k} \geq \beta := \frac{1}{T_k}$, we obtain
\begin{eqnarray*}
\frac{1+rS_{k+1}}{S_{k+1}} &\geq& \frac{1+rS_k}{S_{k+1}} = \frac{S_k}{S_{k+1}}\alpha \stackrel{(\ref{S_k-T_k-ratio})}{=} \frac{2\alpha}{\alpha+2+\sqrt{(\alpha+2)^2-4}}\\
&\geq& \frac{2\beta}{\beta+2+\sqrt{(\beta+2)^2-4}} \stackrel{(\ref{S_k-T_k-ratio})}{=} \frac{T_k}{T_{k+1}}\beta=\frac{1}{T_{k+1}}
\end{eqnarray*}
since $S_{k+1} \geq S_k$ and $x \mapsto \frac{2x}{x+2+\sqrt{(x+2)^2-4}}=\frac{2}{1+2x^{-1}+\sqrt{1+4x^{-1}}}$ is non-decreasing on $(0,\infty)$. Hence, we claim $\frac{1+rS_k}{S_k} \geq \frac{1}{T_k}$ for all $k \geq 0$ and therefore the proof is completed.
\end{proof}

\begin{applem}\label{lem-S_k-ub-nonst}
Let $\{T_k\}_{k \geq 0}$ be a sequence defined by (\ref{S_k-def}) with $r:=0$, namely $T_0:=1$ and $T_{k+1}:=\frac{1+2T_k+\sqrt{1+4T_k}}{2}$ for $k \geq 0$. Then, we have
\[ \frac{\sum_{i=0}^kT_i}{T_k} \leq \frac{1}{3}k +\frac{1}{6}\log(k+2)+1,\quad \forall k \geq 0. \]
\end{applem}

\begin{proof}
The case $k =0$ is obvious. Assume that the assertion is true for some $k \geq 0$. Putting $U_k:=\frac{1}{3}k +\frac{1}{6}\log(k+2)+1$, we have
\[ \frac{\sum_{i=0}^{k+1}T_i}{T_{k+1}} = 1+\frac{T_k}{T_{k+1}}\frac{\sum_{i=0}^kT_i}{T_k} \leq 1+\frac{T_k}{T_{k+1}}U_k. \]
Hence, it remains to show $1+\frac{T_k}{T_{k+1}}U_k \leq U_{k+1}$ for $k \geq 0$.
For that, we analyze the sequence $t_0:=1,~t_{k+1}:=T_{k+1}-T_k$ for $k \geq 0$ (namely, $T_k=\sum_{i=0}^k t_i$). The recurrence relation of $T_k$ implies $t_{k}^2=(T_{k}-T_{k-1})^2=T_{k}$ and
\[ t_{k+1}=T_{k+1}-T_k\stackrel{(\ref{S_k-rec})}{=} \frac{1+\sqrt{1+4T_k}}{2} = \frac{1+\sqrt{1+4t_k^2}}{2},\quad \forall k \geq 0. \]
Analyzing the difference $t_{k+1}-t_k$ shows for $k \geq 0$ that
\[t_{k+1}-t_k=\frac{1+\sqrt{1+4t_k^2}-2t_k}{2}=\frac{1}{2}+\frac{1}{2\left(\sqrt{1+4t_k^2}+2t_k\right)} \leq \frac{1}{2}+\frac{1}{2\left(\sqrt{4t_k^2}+2t_k\right)} = \frac{1}{2}+\frac{1}{8t_k}.\]
Since Lemma \ref{lem-S_k-lower-bound-b} yields $t_k=\sqrt{T_k}\geq \sqrt{(k+1)(k+4)/4}\geq (k+2)/2$ for $k \geq 0$, we obtain
\[ t_{k+1}\leq t_0+\frac{k+1}{2}+\frac{1}{8}\sum_{i=0}^{k}\frac{1}{t_i}\leq \frac{k}{2}+\frac{3}{2}+\frac{1}{8}\sum_{i=0}^{k}\frac{2}{i+2}\leq \frac{k}{2}+\frac{3}{2}+\frac{1}{4}\log(k+2)=\frac{3}{2}U_{k} \]
for all $k \geq 0$.
Finally, this upper bound of $t_k$ concludes that
\[ \frac{U_k}{1+U_k-U_{k+1}}=\frac{3U_k}{2+\frac{1}{2}\log\frac{k+2}{k+3}}\geq \frac{3}{2}U_k \geq t_{k+1}=\frac{t_{k+1}^2}{t_{k+1}}=\frac{T_{k+1}}{T_{k+1}-T_k}. \]
Taking the inverse and multiplying by $U_k$ for both sides yield $1+\frac{T_k}{T_{k+1}}U_k \leq U_{k+1}$.
\end{proof}

} 
\end{document}